\documentclass{amsart}

\usepackage{amsmath,amssymb,mathtools,xcolor}
\usepackage{geometry}
\usepackage{tikz}
\usetikzlibrary{shapes.geometric}

\usepackage{setspace}
\setdisplayskipstretch{1}
\usepackage{microtype}
\usepackage[pdftex,colorlinks,citecolor=blue,linkcolor=blue,urlcolor=blue]{hyperref}
\hypersetup{
  pdftitle={Typical intersecting families at n=2k+1 and n=2k+2},
  pdfsubject={Typical structure of uniform intersecting families}
}

\allowdisplaybreaks
\newtheorem{theorem}{Theorem}[section]
\newtheorem{lemma}[theorem]{Lemma}
\newtheorem{proposition}[theorem]{Proposition}

\newtheorem{conjecture}[theorem]{Conjecture}
\newtheorem*{bglwconjecture}{Conjecture}
\theoremstyle{definition}
\newtheorem{claim}[theorem]{Claim}
\newtheorem{definition}[theorem]{Definition}

\newcommand{\cC}{\mathcal C}

\newcommand{\cF}{\mathcal F}
\newcommand{\cG}{\mathcal G}
\newcommand{\cI}{\mathcal I}
\newcommand{\cS}{\mathcal S}
\newcommand{\cW}{\mathcal W}
\newcommand{\eps}{\varepsilon}
\newcommand{\1}{\mathbf 1}
\DeclareMathOperator{\KG}{KG}

\title{Typical intersecting families at $n=2k+1$ and $n=2k+2$}
\author{Lina Li$^\dagger$}
\address{$^\dagger$Department of Mathematics, University of Mississippi}
\email{lli5@olemiss.edu}
\thanks{Lina Li is supported by the Simons Foundation SFI-MPS-TSM-00024322.}
\date{}

\newcommand{\dv}{\operatorname{div}}
\newcommand{\cP}{\mathcal P}
\newcommand{\ind}{\mathrm{ind}}
\newcommand{\val}{\mathrm{val}}

\begin{document}
\begin{abstract}
A family of sets is \emph{intersecting} if every two members intersect, and \emph{trivial} if all its members contain a common element.
We determine the typical structure of $k$-uniform intersecting families on $2k+1$ and $2k+2$ elements as $k\to\infty$.
For $n=2k+2$, we prove that almost all intersecting families are trivial and that their number is
\[
(2k+2+o(1))\,2^{\binom{2k+1}{k-1}}.
\]
Together with Yang's recent result for $n\ge 2k+3$, this settles a conjecture of Balogh, Garcia, Li, and Wagner.
For $n=2k+1$, almost all intersecting families are nontrivial. We prove that, as conjectured by the same authors, a typical intersecting family is close to a full star: its members outside the star form
components of size at most two in the graph joining sets that intersect in $k-1$ elements.
We also obtain an asymptotic formula for the number of intersecting families in this case, with an explicit second-order term in the exponent.
Our proof combines Sapozhenko's graph container method and stability in Kneser graphs to control families far from every star, and a polymer model and cluster expansion to enumerate families close to a fixed star.
\end{abstract}
\maketitle

\section{Introduction}
\subsection{Background}
A family $\cF\subseteq\binom{[n]}k$ is \emph{intersecting} if every pair of members of $\cF$ has a non-empty intersection. 
An intersecting family is \emph{trivial}, or a \emph{star}, if all its members contain a fixed element, called a \emph{center}, and \emph{nontrivial} otherwise. For $x\in[n]$, the \emph{full star centered at $x$} is defined as
\[
 \cS_x:=\left\{F\in\binom{[n]}k:x\in F\right\}.
\]
The seminal result of Erd\H{o}s--Ko--Rado~\cite{ErdosKoRado1961} states that, for $n\ge2k$, every intersecting family $\cF\subseteq\binom{[n]}k$ satisfies
\[
 |\cF|\le\binom{n-1}{k-1}.
\]
This inequality is best possible, as shown by full stars; for $n>2k$, equality holds only for full stars. 
Hilton and Milner~\cite{HiltonMilner1967} later determined the maximum size of a nontrivial intersecting family. 

Intersecting families are also closely related to another important concept in combinatorics---\emph{Kneser graphs}.
The Kneser graph $\KG(n,k)$ has vertex set $\binom{[n]}k$,
with two vertices adjacent if and only if the corresponding
sets are disjoint. Thus intersecting families are precisely
the \emph{independent sets} of $\KG(n,k)$.
This correspondence allows us to use graph methods to study intersecting families and plays a crucial role throughout the paper.

Beyond these extremal results, it is natural to ask how many intersecting families there are and whether they exhibit any typical structure.
Balogh, Das, Delcourt, Liu, and Sharifzadeh~\cite{BaloghDasDelcourtLiuSharifzadeh2015} first determined the logarithmic asymptotics of the number of intersecting families for any $n\ge2k+1$. 
Later, Frankl--Kupavskii~\cite{FranklKupavskii2018} and
Balogh--Das--Liu--Sharifzadeh--Tran~\cite{BaloghDasLiuSharifzadehTran2019}
independently proved that almost all intersecting families are trivial when $n\ge 2k+C\sqrt{k\log k}$ for a sufficiently large absolute constant $C$.
Balogh, Garcia, Li, and Wagner~\cite{BaloghGarciaLiWagner2024} (abbreviated as BGLW) extended this result to $n\ge2k+100\ln k$
and conjectured that almost all intersecting families are trivial for every $n\ge2k+2$.

\begin{conjecture}[Balogh--Garcia--Li--Wagner~\cite{BaloghGarciaLiWagner2024}]\label{conj:BGLW1}
For every $n\ge2k+2$, almost all intersecting families in $\binom{[n]}k$ are trivial. More precisely, as $k\to\infty$, the number of such intersecting families is
\[
 (n+o(1))2^{\binom{n-1}{k-1}}.
\]
\end{conjecture}

For $n=2k+1$, the behavior is different.
A construction of Balogh--Das--Liu--Sharifzadeh--Tran~\cite{BaloghDasLiuSharifzadehTran2019},
further analyzed by BGLW~\cite{BaloghGarciaLiWagner2024},
shows that almost all intersecting families are nontrivial.
BGLW conjectured that a typical intersecting family nevertheless has a simple local structure: it is \emph{close} to a full star, in the sense that its members outside the star occur only as
isolated defects or as adjacent pairs, with no larger connected
groups. (Here two sets are adjacent if they intersect in $k-1$
elements.) They call such a family \emph{nice}. 
The precise definition of niceness requires more background, so we defer it until later; see Definition~\ref{def:nice}.

\begin{bglwconjecture}[Balogh--Garcia--Li--Wagner~\cite{BaloghGarciaLiWagner2024}]
Almost all intersecting families $\cF\subseteq\binom{[2k+1]}k$ are nice. In fact, the number of intersecting families is
\[
 2^{\binom{2k}{k-1}}
 \exp\left(\binom{2k}k2^{-k}+\Theta(k^{3/2})\right).
\]
\end{bglwconjecture}

Very recently, Yang~\cite{Yang2026} proved Conjecture~\ref{conj:BGLW1} for every $n\ge2k+3$, leaving only the cases $n=2k+2$ and $n=2k+1$ unresolved.

\subsection{Main results}
In this paper, we resolve the two remaining cases $n=2k+2$ and $n=2k+1$, confirming both conjectures of BGLW~\cite{BaloghGarciaLiWagner2024}.

Before stating our results, we first count the trivial families to establish a \emph{benchmark} for counting all intersecting families. Each of the $n$ full stars in $\binom{[n]}k$ contains $2^{\binom{n-1}{k-1}}$ intersecting families, while any two distinct full stars contain $2^{\binom{n-2}{k-2}}$ families in common. Hence the number of trivial intersecting families is
\begin{equation}\label{eq:trivial-benchmark}
 n2^{\binom{n-1}{k-1}}-O\left(n^2 2^{\binom{n-2}{k-2}}\right)
 =(n+o(1))2^{\binom{n-1}{k-1}},
\end{equation}
where the error term is negligible for any $n\ge k+2$ because $\binom{n-1}{k-1}-\binom{n-2}{k-2}=\binom{n-2}{k-1}$. 
The number of trivial families provides a natural baseline for the asymptotic formulas in our main results below. 

Our first result shows that almost all intersecting families are trivial when $n=2k+2$, completing the proof of Conjecture~\ref{conj:BGLW1}.
\begin{theorem}\label{thm:main}
For $n=2k+2$, as $k\to\infty$, the number of intersecting families in $\binom{[2k+2]}k$ is
\[
 (2k+2+o(1))2^{\binom{2k+1}{k-1}}.
\]
In particular, almost all intersecting families in $\binom{[2k+2]}k$ are trivial. 
\end{theorem}

For $n=2k+1$, we prove that almost all intersecting families are nice and obtain a sharper asymptotic formula for their number than conjectured by BGLW.

\begin{theorem}\label{thm:odd}
For $n=2k+1$, as $k\to\infty$, the number of intersecting families in $\binom{[2k+1]}k$ is
\[
 (2k+1)2^{\binom{2k}{k-1}}\exp\left(\binom{2k}k2^{-k}+\frac12\binom{2k}k(k^2-2)2^{-2k}+O(k^{17} 2^{-k})\right).
\]
Moreover, almost all intersecting families in $\binom{[2k+1]}k$ are nice. 
\end{theorem}

As noted above, counting intersecting families is equivalent to counting independent sets in Kneser graphs.
The study of independent sets has a rich history and has led to powerful general tools, notably the hypergraph
container method of Balogh, Morris, and Samotij~\cite{balogh2015independent} and Saxton and Thomason~\cite{saxton2015hypergraph}.
More closely related to our work, Jenssen and Perkins~\cite{JenssenPerkins2020} and Balogh, Garcia,
and Li~\cite{BaloghGarciaLi2021} combine \emph{Sapozhenko's graph container method} with the \emph{cluster expansion} to
count independent sets in the hypercube and in the middle two layers of the Boolean lattice, respectively.
In particular, relative to a fixed star, intersecting families can be encoded as independent sets in a bipartite
graph similar to the one studied by Balogh, Garcia, and Li~\cite{BaloghGarciaLi2021}, with an additional
intersecting condition on one side.
Dropping this condition allows us to obtain upper bounds by adapting their counting argument.
To obtain the asymptotic formula, however, the main obstacle is to establish that almost all intersecting
families are close to some full star, which requires exploiting the structure of intersecting families
and Kneser graphs.

To overcome this obstacle, we combine Sapozhenko's graph container method with stability arguments in the Kneser
graph, where the intersecting condition is exploited.
A key new ingredient is a way to transfer Sapozhenko's certificates for local structures into global Kneser
containers, yielding the sharper counting estimates needed for these two boundary cases.
For $n=2k+1$, we further adapt polymer models and the cluster expansion to obtain the asymptotic formula
and determine the typical structure. We give a detailed overview in Section~\ref{sec:overview}.

\subsection{Notation and preliminaries}\label{sec:notation}
For $n\in\mathbb N$, we let $[n]:=\{1,\ldots,n\}$ and denote by
$\binom{[n]}k$ the collection of its $k$-element subsets.
For sets $A,B$, write
$A\mathbin\triangle B=(A\setminus B)\cup(B\setminus A)$
for their symmetric difference.

Throughout the paper, $\log$ denotes the binary logarithm and
$\ln$ denotes the natural logarithm. All asymptotic notation refers
to $k\to\infty$.
For functions $f,g:\mathbb N\to(0,\infty)$, we write $f\sim g$
if $f/g\to1$ and $f=o(g)$ if $f/g\to0$.
We also write $f=O(g)$ if $f\le Cg$ and $f=\Omega(g)$ if $f\ge cg$
for all sufficiently large $k$, where $C,c>0$ are constants.
The constants in $O(\cdot)$ and $\Omega(\cdot)$ are absolute
unless a subscript indicates dependence on a parameter.
We omit floor and ceiling
signs when they are not crucial. 

\medskip
\noindent\textit{Parameters for intersecting families.}
For $n>2k$, write $r:=n-2k$ and, for simplicity, let
\[
 M:=\binom{n-1}{k-1},\qquad
 N:=\binom{n-1}k,\qquad
 \rho:=\frac NM=1+\frac rk,\qquad
 d:=\binom{k+r-1}{k-1}.
\]
Here $M$ is the size of a full star, that is, the maximum size of an intersecting family by the Erd\H{o}s--Ko--Rado theorem.
For a fixed center $x$, $N$ is the number of $k$-sets outside $\cS_x$, namely, those that do not contain $x$.
For every $k$-set $F$ not containing $x$, there are exactly $d$
members of $\cS_x$ disjoint from $F$; thus, including $F$ in an
intersecting family forbids exactly $d$ members of $\cS_x$.
Equivalently, $d$ is the degree on one side of the bipartite inclusion graph introduced in Section~\ref{sec:inclusion}.

Unless stated otherwise, we assume that
$n\in\{2k+1,2k+2\}$, in which case
\[
 d=
 \begin{cases}
 k,&n=2k+1,\\[2pt]
 \dfrac{k(k+1)}2,&n=2k+2.
 \end{cases}
\]

\medskip
\noindent\textit{Graph notation and connected sets.}
For a graph $G$ and a vertex $v\in V(G)$, let $N_G(v)$ be the
\emph{neighborhood} of $v$ and let $d_G(v):=|N_G(v)|$ be its
\emph{degree}. For $W\subseteq V(G)$, put
$d_W(v):=|N_G(v)\cap W|$.
Let $\Delta(G):=\max_{v\in V(G)}d_G(v)$ denote the
\emph{maximum degree} of $G$.
For $W\subseteq V(G)$, let $G[W]$ be the subgraph of $G$
\emph{induced} by $W$, write $e_G(W)$ for the number of edges
in $G[W]$, and let
$N_G(W):=\bigcup_{v\in W}N_G(v)$.
We omit the subscript $G$ when the graph is clear.

We will use a standard bound on the number of connected subsets.
It follows from the fact that the infinite $\Delta$-branching rooted tree contains exactly
\[
 \frac{1}{(\Delta-1)s+1}\binom{\Delta s}{s}
 \le(e\Delta)^{s-1}
\]
rooted subtrees with $s$ vertices; see, e.g.,~\cite[p.~396, Ex.~11]{Knuth1969}.

\begin{proposition}\label{prop:connected-sets}
Let $G$ be a graph with maximum degree at most $\Delta\ge1$.
For every integer $s\ge1$, the number of $s$-vertex subsets
of $V(G)$ containing a specified vertex and inducing a connected
subgraph is at most $(e\Delta)^{s-1}$.
\end{proposition}

Let $\Sigma$ be a bipartite graph with classes $X$ and $Y$. Say that $A\subseteq X$ is \emph{2-linked} if $\Sigma^2[A]$ is connected, where $\Sigma^2$ is the square of $\Sigma$ (that is, two distinct vertices are adjacent in $\Sigma^2$ whenever their distance in $\Sigma$ is at most two). 
A \emph{2-linked component} of $W\subseteq X$ is a maximal
2-linked subset of $W$.
For $A\subseteq X$, let
\[
 [A]_\Sigma:=\{u\in X:N_\Sigma(u)\subseteq N_\Sigma(A)\}
\]
be the \emph{closure} of $A$.
We omit the subscript $\Sigma$ when the underlying bipartite
graph is clear.

\medskip
\noindent\textit{Binomial estimates.}
For real numbers $0\le h\le m$, we use
$\binom{m}{\le h}$ as shorthand for
$\sum_{0\le j\le h}\binom{\lfloor m\rfloor}{j}$,
where the sum is over integers $j$.
The \emph{binary entropy function} $H_2:[0,1]\to\mathbb R$
is defined by
\[
 H_2(p):=-p\log p-(1-p)\log(1-p)
 \qquad (0<p<1),
\]
with $H_2(0)=H_2(1):=0$ by convention.
We will use the following standard binomial estimates.

\begin{proposition}\label{lem:binomial-estimates}
Let $m$ be a positive integer. For every $0<p\le1/2$,
\[
 \binom m{\le pm}\le2^{H_2(p)m}.
\]
For any integers $1\le h\le m$,
\[
 \binom m{\le h}\le\left(\frac{em}{h}\right)^h.
\]
\end{proposition}

\section{Proof overview and key lemmas}\label{sec:overview}
In this section, we provide an overview of our proof.
Our goal is to explain the main difficulties in the problem and highlight the new ideas needed to overcome them, without focusing on technical details.
We then state two key lemmas and deduce our main theorems from them.

\subsection{General strategy} For $\xi\in(0,1)$ and $x\in[n]$, we say that an intersecting family $\cF\subseteq\binom{[n]}k$ is in the \emph{star phase centered at $x$} if
\[
 |\cF\setminus\cS_x|\le\xi M.
\]
Thus every intersecting family either lies in one of the star phases or is \emph{far from looking like a single star}. 
Note that a family may lie in more than one star phase, and no canonical choice will be needed.

Informally, the general strategy for counting intersecting families consists of the following three steps:

\medskip\noindent\textbf{Step 1.}
Show that almost all intersecting families lie in one of the $n$ star phases;

\noindent\textbf{Step 2.} 
For every $x\in[n]$, count the intersecting families in the star phase centered at $x$
and determine their typical structure;

\noindent\textbf{Step 3.} 
Combine the first two steps to obtain the total count and typical structure of intersecting families.
\medskip

This general strategy of first showing that almost every object lies in one of the \emph{phases}, and then counting the objects in each fixed phase, has appeared in many enumeration problems. These include independent sets and proper colorings (see, e.g.,~\cite{JenssenPerkins2020,KahnPark2020,LiMcKinleyPark2025}), as well as related objects such as maximal independent sets (see, e.g.,~\cite{BaloghChenGarcia2025,KahnPark2022}). For intersecting families, this strategy was exploited by BGLW~\cite{BaloghGarciaLiWagner2024} (although their measure of deviation from a full star is based on the closure size of the non-star part) and by Yang~\cite{Yang2026}.

For \textbf{Step~2}, both Yang's proof~\cite{Yang2026} and ours follow the counting framework developed by BGLW~\cite[Section~4]{BaloghGarciaLiWagner2024}.
Briefly, for a fixed $x\in[n]$, we encode each intersecting family in the star phase centered at $x$ as an independent set in a bipartite inclusion graph (see Lemma~\ref{lem:encoding} for details) and then count these independent sets using the graph container method of Sapozhenko~\cite{Sapozhenko1987}, together with some isoperimetric properties of the bipartite inclusion graph. Sapozhenko's method is particularly well suited to counting independent sets in the bipartite setting, but it requires certain expansion properties, which can be achieved here
by restricting to a fixed phase.

As $n$ approaches $2k$, however, including a set outside the star becomes less costly: it excludes only $d=\binom{n-k-1}{k-1}$ members of the star, and $d$ decreases as $n$ decreases. Thus, showing that nontrivial families are negligible requires increasingly precise estimates.
Yang~\cite{Yang2026} refined the estimates within BGLW's framework to reach $n\ge2k+3$, and we further refine them for $n=2k+2$.

For $n=2k+1$, the situation changes: the saving from excluding $d=k$ star members is insufficient to compensate for the number of possible choices for a set outside the star, and nontrivial
families dominate.
We must therefore count the contributions of these outside sets and their interactions more precisely.
To do so, we use \emph{polymer models} and the
\emph{cluster expansion}, tools from statistical physics that have proved effective in recent work on asymptotic enumeration (see, e.g.,~\cite{JenssenPerkins2020, LiMcKinleyPark2025}).
These tools allow us to obtain the asymptotic formula and show that almost all families in a fixed phase are nice with the given center.

Nevertheless, since the work of BGLW, the main obstacle
to extending the result to smaller $n$ has been \textbf{Step~1}.
BGLW~\cite{BaloghGarciaLiWagner2024} control the exceptional families by combining the Das--Tran removal lemma with a bound on the number of maximal intersecting families derived from Bollob\'as's set-pairs inequality, but this argument requires $n\ge2k+100\ln k$.
To reach $n\ge2k+3$, Yang~\cite{Yang2026} constructs containers for the Kneser graph using an argument in the spirit of the Kleitman--Winston method~\cite{KleitmanWinston1982} (a precursor to the seminal hypergraph container method~\cite{balogh2015independent, saxton2015hypergraph}).
Yang then applies spectral supersaturation and the Das--Tran
removal lemma to show that containers large enough to contribute significantly must be \emph{close} to a full star. This reduces the problem to counting within fixed star phases.

For fixed $r=n-2k$, Yang's bound on the number of containers
is $\exp(O(M\log k/k^{r-1}))$, whereas the saving is of order $\exp(\Theta(M/k))$.
When $n=2k+2$, this leaves a logarithmic gap in the exponent.
We can overcome this gap by combining Yang's argument
with a more detailed degree analysis in the spirit of
Kahn and Park~\cite[Section~3]{KahnPark2022}, extending the
result to $n=2k+2$. However, neither Yang's original argument
nor this refinement gives a sufficiently small enumeration
cost for $n=2k+1$, so a different approach is needed.

Our approach instead uses Sapozhenko's container method
to obtain sharper container estimates through the bipartite
inclusion graph already used in Step~2. 
Sapozhenko's container method generally gives finer estimates
on bipartite graphs than the Kleitman--Winston approach,
which applies to general graphs.
The main idea is to construct container certificates in the bipartite inclusion graph and transfer them to containers in the Kneser graph, where supersaturation and the Das--Tran
removal lemma can again be applied.

Making this transfer requires additional ideas.
Sapozhenko's method provides container certificates for the \emph{local} $2$-linked components of the defect set
(the part outside the star), whereas the stability argument requires a \emph{global} container for independent sets of the
Kneser graph.
To bridge this gap, we join all the large components into a single $2$-linked set by adding a small number of vertices, so that they can be captured by one certificate.
When the resulting set has strong expansion, a standard container reconstruction argument gives the required
counting bound. Otherwise, we transfer its certificate to a Kneser graph container of small maximum degree
and use stability to bound the number of exceptional families it contains.
The small components are handled separately by isoperimetry and elementary counting.

\subsection{Key lemmas} Next, we describe more details for each step. For a family $\cF\subseteq\binom{[n]}k$, define its \emph{diversity} by
\[
 \dv(\cF):=\min_{x\in[n]}|\cF\setminus\cS_x|.
\]
Thus $\dv(\cF)$ is the minimum number of members outside a full star. The first step is the following global phase selection lemma for the two boundary cases.

\begin{lemma}[Step 1: global phase selection]\label{thm:global}
There are absolute constants $C_1,c_1>0$ such that the following holds. For sufficiently large $k$, any integer $n\in\{2k+1,2k+2\}$, and any real number
\[
 \xi_1\ge C_1(\log k)^2/d,
\]
the number of intersecting families $\cF\subseteq\binom{[n]}k$ satisfying
\[
 \dv(\cF)>\xi_1M
\]
is at most $2^{M-c_1\xi_1M/k}=o(2^M)$.
\end{lemma}

Here, the range of $\xi_1$ is the best we can achieve with the current argument. We made no attempt to optimize the constants; indeed, for our main theorems, a sufficiently small positive constant $\xi_1$ would suffice.

Lemma~\ref{thm:global} shows that all but $o(2^M)$ intersecting families lie in at least one star phase.
We then focus on a fixed star phase and estimate the number of nontrivial intersecting families within it. The estimate takes different forms in the two cases.

\begin{lemma}[Step 2: counting within a fixed star phase]\label{thm:local}
There are absolute constants $C_2,c_2>0$ such that the following holds. For every fixed constant $\xi_2$ satisfying
\[
 0<\xi_2\le C_2,
\]
all sufficiently large $k$, every $n\in\{2k+1,2k+2\}$, and every $x\in[n]$, the number of nontrivial intersecting families $\cF\subseteq\binom{[n]}k$ satisfying
\[
 |\cF\setminus\cS_x|\le\xi_2M
\]
is
\[
\begin{cases}
 2^M\exp\left(N2^{-k}+\dfrac N2(k^2-2)2^{-2k}
 +O(k^{17} 2^{-k})\right),
   & \text{if } n=2k+1,\\[4pt]
 \text{at most } 2^{M-c_2k^2},
   & \text{if } n=2k+2.
\end{cases}
\]
When $n=2k+1$, the same estimate holds if the count is restricted to families that are nice with center $x$; see Definition~\ref{def:nice}.
\end{lemma}


\subsection{Proofs of Theorems~\ref{thm:main} and~\ref{thm:odd}} We now carry out Step 3 and derive the main theorems from the two lemmas.

\begin{proof}[Proof of Theorem~\ref{thm:main}]
Let $n=2k+2$, and fix a constant $\xi\in(0,C_2]$, independently of $k$. Since $C_1(\log k)^2/d\to0$, for all sufficiently large $k$ we have
\[
 C_1\frac{(\log k)^2}{d}\le\xi\le C_2.
\]

By Lemmas~\ref{thm:global} and~\ref{thm:local}, the number of nontrivial intersecting families is at most
\[
 o(2^M)+n2^{M-c_2k^2}=o(2^M).
\]
Combining this bound with the trivial family benchmark~\eqref{eq:trivial-benchmark} proves Theorem~\ref{thm:main}.
\end{proof}

\begin{proof}[Proof of Theorem~\ref{thm:odd}]
Let $n=2k+1$, and fix a constant $\xi\in(0,C_2]$ (independently of $k$) such that $3H_2(2\xi/3)<1/8$. Since $C_1(\log k)^2/d\to0$, we also have $C_1(\log k)^2/d\le\xi$ for all sufficiently large $k$. 
For simplicity, we define
\[
 \Phi_k:=N2^{-k}+\frac N2(k^2-2)2^{-2k}.
\]
Note that $\Phi_k\sim N2^{-k}\sim2^k/\sqrt{\pi k}$.

For every $x\in[n]$, let
\[
 \mathcal I_x:=\left\{\cF\subseteq\binom{[n]}k:
 \cF\text{ is intersecting and }|\cF\setminus\cS_x|\le\xi M\right\}.
\]
By Lemma~\ref{thm:global}, the total number of intersecting families satisfies
\begin{equation}\label{eq:totalcount}
 \sum_{x\in[n]}|\mathcal I_x|
 -\sum_{\{x,y\}\in\binom{[n]}2}|\mathcal I_x\cap\mathcal I_y|
 \le
 \#\left\{\cF\subseteq\binom{[n]}k:\cF\text{ is intersecting}\right\}
 \le
 \sum_{x\in[n]}|\mathcal I_x|
 + o(2^M).
\end{equation}
Note that each $\mathcal I_x$ contains at most $n2^M$ trivial families. This, together with Lemma~\ref{thm:local}, gives
\[
 \sum_{x\in[n]}|\mathcal I_x|
 =n2^M\left(\exp\bigl(\Phi_k+O(k^{17} 2^{-k})\bigr)+O(n)\right)
 =n2^M\exp\bigl(\Phi_k+O(k^{17} 2^{-k})\bigr).
\]

For distinct $x,y$, any family $\mathcal{F}\in \mathcal I_x\cap\mathcal I_y$ has at most $2\xi M$ members outside $\cS_x\cap\cS_y$. Since $|\cS_x\cap\cS_y|=\binom{n-2}{k-2}$ and $\binom nk=(n/k)M\le3M$,
\[
 |\mathcal I_x\cap\mathcal I_y|
 \le 2^{\binom{n-2}{k-2}}\sum_{j\le2\xi M}\binom{3M}j
 \le 2^{\binom{n-2}{k-2}}\cdot 2^{H_2(2\xi/3)\cdot 3M}
 \le 2^{3M/4},
\]
where the second inequality uses Proposition~\ref{lem:binomial-estimates}, and the last inequality follows from the choice of $\xi$ and $\binom{n-2}{k-2}/M=(k-1)/(n-1)<1/2$. Hence we have
\[
 \sum_{\{x,y\}\in\binom{[n]}2}|\mathcal I_x\cap\mathcal I_y|
 \le\binom n2 2^{3M/4}=o(2^M).
\]
Substituting this into~\eqref{eq:totalcount}, we conclude that the total number of intersecting families is
\[
  \sum_{x\in[n]}|\mathcal I_x| + o(2^M)= n2^M\exp\bigl(\Phi_k+O(k^{17} 2^{-k})\bigr),
\]
which proves the first part of Theorem~\ref{thm:odd}.

It remains to show that almost all intersecting families are nice.
By Lemma~\ref{thm:local}, the number of nontrivial families in each
$\mathcal I_x$ that are not nice is $o(|\mathcal I_x|)$.
Since every trivial family is nice, Lemma~\ref{thm:global} gives
\[
 \#\left\{\cF\subseteq\binom{[n]}k:
 \cF\text{ is intersecting but not nice}\right\}
 \le o(2^M)+o\left(\sum_{x\in[n]}|\mathcal I_x|\right),
\]
which is negligible compared with the total
number of intersecting families.
\end{proof}

\subsection{Organization.}
In Section~\ref{sec:containers}, we explain how to encode intersecting families as independent sets in a certain bipartite graph, and establish containers and counting estimates for this graph using Sapozhenko’s container method. These estimates are key ingredients in the proofs of both Lemmas~\ref{thm:global} and~\ref{thm:local}. 
Then, in Section~\ref{sec:global}, we prove the global phase selection lemma, Lemma~\ref{thm:global}. In Section~\ref{sec:polymer-background}, we present some preliminaries on abstract polymer models and the cluster expansion. Finally, in Section~\ref{sec:local}, we prove Lemma~\ref{thm:local}, counting intersecting families in a fixed star phase by combining Sapozhenko's container method with the cluster expansion.

\section{Encoding and containers}\label{sec:containers}

\subsection{The bipartite inclusion graph}\label{sec:inclusion}
Following BGLW~\cite{BaloghGarciaLiWagner2024}, we encode intersecting families using the bipartite graph defined below.
\begin{definition}[Bipartite inclusion graph]\label{def:inclusion-graph}
For $x\in[n]$, the \emph{bipartite inclusion graph} $H$ associated
with the center $x$ has parts
\[
 \mathcal U:=\binom{[n]\setminus\{x\}}{k+r-1},
 \qquad
 \mathcal L:=\binom{[n]\setminus\{x\}}{k-1},
\]
where two sets are adjacent if and only if one contains the other.
\end{definition}
Note that the degrees on $\mathcal U,\mathcal L$ are $d,\rho d$, respectively, and $|\mathcal U|=N=\rho M$, $|\mathcal L|=M$. 
By symmetry, the choice of \(x\) does not affect the properties of \(H\), so we omit \(x\) from the notation for $H$ and its parts
throughout the paper.

We will use the following isoperimetric bounds on $H$, which follow directly from Lov\'asz's simplified formulation of the Kruskal--Katona theorem.
\begin{proposition}\label{prop:isoperimetry}
Let $\varnothing\ne A\subseteq\mathcal U$. If
$|A|=\binom{y}{k+r-1}$ for some real $y\ge k+r-1$, then
\begin{equation}\label{eq:shadow-ratio}
 \frac{|N_H(A)|}{|A|}
 \ge\prod_{i=0}^{r-1}\frac{k+i}{y-k+1-i}.
\end{equation}
In particular, for sufficiently large $k$, the following hold.
\begin{enumerate}
\item[(i)] If $|A|\le k^3$, then
$|N_H(A)|\ge d|A|/\binom{r+4}{r}$.
\item[(ii)] If $|A|\le2k^{12}$, then
$|N_H(A)|\ge d|A|/200$.
\end{enumerate}
\end{proposition}
\begin{proof}
Since $N_H(A)$ is the $(k-1)$-shadow of $A$,
the Kruskal--Katona theorem (see~\cite{lovasz2007combinatorial}) gives
\[
 \frac{|N_H(A)|}{|A|}
 \ge\frac{\binom{y}{k-1}}{\binom{y}{k+r-1}}
 =\prod_{i=0}^{r-1}\frac{k+i}{y-k+1-i}
 =\frac{d}{\binom{y-k+1}{r}}.
\]
Then (i) and (ii) follow from~\eqref{eq:shadow-ratio}
using $y\le k+r+3$ and $y\le k+r+12$, respectively, for sufficiently large $k$.
\end{proof}

We now show how to encode intersecting families as independent sets in $H$.
For a family $\cF\subseteq\binom{[n]}k$, define
\[
 A_x(\cF):=\{[n]\setminus(F\cup\{x\}):F\in\cF,\ x\notin F\},\qquad
 B_x(\cF):=\{F\setminus\{x\}:F\in\cF,\ x\in F\}.
\]
Observe that \(A_x(\cF)\subseteq\mathcal U\) and
\(B_x(\cF)\subseteq\mathcal L\), so
\(A_x(\cF)\cup B_x(\cF)\subseteq V(H)\).
Intuitively, \(A_x(\cF)\) corresponds to the members of \(\cF\) outside the full star \(\cS_x\), while \(B_x(\cF)\) corresponds to those inside \(\cS_x\).
We say that a family of sets \(\cF\) is \emph{\(r\)-intersecting} if every two of its members have intersection of size at least \(r\).

\begin{lemma}\label{lem:encoding}
Let $n=2k+r$ with $r\in\{1,2\}$, and fix $x\in[n]$.
There is a bijection between intersecting families in \(\binom{[n]}k\) and independent sets \(I\) of \(H\) with \(I\cap\mathcal U\) being \(r\)-intersecting, given by
\[
 \cF\longmapsto A_x(\cF)\cup B_x(\cF).
\]
In particular, for each fixed $r$-intersecting $A\subseteq\mathcal U$, there are exactly $2^{M-|N_H(A)|}$ intersecting families $\cF\subseteq\binom{[n]}k$ satisfying $A_x(\cF)=A$.
\end{lemma}
\begin{proof}
Let $\cF\subseteq\binom{[n]}k$ be intersecting. For distinct
$u,u'\in A_x(\cF)\subseteq \mathcal{U}=\binom{[n]\setminus\{x\}}{k+r-1}$, the corresponding members of $\cF$ satisfy
\begin{equation}\label{ieq:encode}
 \bigl|([n]\setminus(u\cup\{x\}))
 \cap([n]\setminus(u'\cup\{x\}))\bigr|
 =n-1 - |u\cup u'|
 =|u\cap u'|-(r-1)\ge1.
\end{equation}
Thus $A_x(\cF)$ is always $r$-intersecting. 
Moreover, \(A_x(\cF)\cup B_x(\cF)\) is independent in \(H\); otherwise, there exist \(u\in A_x(\cF)\) and \(v\in B_x(\cF)\) adjacent in \(H\). Then, by the definition of \(H\), we have \(v\subseteq u\), and consequently the corresponding members of \(\cF\) are disjoint, contradicting the assumption that \(\cF\) is intersecting.
Hence, the map is well-defined.
The map is injective because $\cF$ can be uniquely recovered from
$A_x(\cF)$ and $B_x(\cF)$.

To prove surjectivity, let $I$ be an independent set of $H$ with $A:=I\cap\mathcal U$ and  $B:=I\cap\mathcal L$, such that
$A$ is $r$-intersecting.
Define
\[
 \cF:=\{[n]\setminus(u\cup\{x\}):u\in A\}
 \cup\{\{x\}\cup v:v\in B\}\subseteq \binom{[n]}{k}.
\]
The members of \(\cF\) containing \(x\) intersect automatically.
For any \(S,S'\in\cF\) not containing \(x\), applying \(\eqref{ieq:encode}\) gives \(|S\cap S'|\ge r-(r-1)=1\), where the inequality follows from the fact that \(A\) is \(r\)-intersecting. Finally, independence of \(I\) ensures that no \(v\in B\) is contained in any \(u\in A\), and thus that every member of \(\cF\) containing \(x\) intersects every member not containing \(x\). Hence \(\cF\) is intersecting, and \(I=A_x(\cF)\cup B_x(\cF)\).

For each fixed \(r\)-intersecting \(A\subseteq\mathcal U\), this bijection indicates that a family \(\mathcal F\) with \(A_x(\mathcal F)=A\) is intersecting if and only if \(B_x(\mathcal F)\subseteq\mathcal L\setminus N_H(A)\). There are therefore exactly \(2^{M-|N_H(A)|}\) intersecting families with \(A_x(\mathcal F)=A\).
\end{proof}

\subsection{Sapozhenko's graph container lemma}
In this subsection, we state a biregular version of a graph container lemma that originates in Sapozhenko's work~\cite{Sapozhenko1987}; see Galvin~\cite{Galvin2019} for an exposition. The formulation stated here is due to BGLW~\cite[Theorem~5.1]{BaloghGarciaLiWagner2024}.
In the statement below, neighborhoods, closures, and 2-linkedness are all taken with respect to the general bipartite graph $\Sigma$.

\begin{lemma}[\cite{BaloghGarciaLiWagner2024}]\label{thm:BGLW-container}
Let \(\Sigma\) be a bipartite graph with parts \(X\) and \(Y\), where every vertex in \(X\) has degree \(q\) and every vertex in \(Y\) has degree \(s\).
For integers $a,g$, let
\[
 \cG_{\Sigma}(a,g):=\{A\subseteq X:A\text{ is a 2-linked set},\ |[A]|=a,\ |N(A)|=g\},
\]
and set $t:=gs-aq$. Let $1\le\varphi\le s-1$, $1\le\psi\le\min\{q,s\}-1$, and let $C>0$ satisfy $C\ln q/(\varphi q)<1$. Put
\[
 m_\varphi:=\min\{|N(K)|:v\in Y,\ K\subseteq N(v),\ |K|>\varphi\}.
\]
Then there exists a family of containers $\cW_{\Sigma}(a,g)\subseteq2^X\times2^Y$ with
\begin{equation}\label{eq:certgeneral}
\begin{aligned}
 |\cW_{\Sigma}(a,g)|\le{}&|Y|\exp\left(
 \frac{54Cg\ln q\ln(qs)}{\varphi q}
 +\frac{54g\ln(qs)}{q^{Cm_\varphi/(\varphi q)}}
 +\frac{54t\ln s\ln(qs)}{q(s-\varphi)}\right)\\
 &\quad\cdot
 \binom{\frac{3Cgs\ln q}{\varphi q}}{\le\frac{3Ct\ln q}{\varphi q}}
 \binom{gs}{\le t/((s-\varphi)\psi)}
 \binom{gsq}{\le t/((q-\psi)\psi)},
\end{aligned}
\end{equation}
and a function $f:\cG_{\Sigma}(a,g)\to\cW_{\Sigma}(a,g)$ such that, for each $A\in\cG_{\Sigma}(a,g)$, writing $(S,F):=f(A)$, the following hold:
\begin{enumerate}
 \item[(i)] $S\supseteq[A]$ and $F\subseteq N(A)$;
 \item[(ii)] $d_F(u)\ge q-\psi$ for every $u\in S$;
 \item[(iii)] $d_{X\setminus S}(v)\ge s-\psi$ for every $v\in Y\setminus F$;
 \item[(iv)] $\displaystyle |S|\le\frac sq|F|+\frac{\psi t}{q}
 \left(\frac1{q-\psi}+\frac1{s-\psi}\right)$.
\end{enumerate}
\end{lemma}
Note that \(t=gs-aq\) counts the edges between \(N(A)\) and \(X\setminus[A]\), so only choices of \(a,g\) with \(t\ge0\) are \emph{feasible}.

\subsection{Containers for the bipartite inclusion graph}
We first specialize Sapozhenko's container lemma to the bipartite inclusion graph \(H\) defined in Definition~\ref{def:inclusion-graph}.
Recall that \(H\) has two parts, \(\mathcal U\) and \(\mathcal L\), whose vertices have degrees \(d\) and \(\rho d\), respectively.
For feasible integers $a,g\ge1$, we have
\[
 \cG_H(a,g)=\{A\subseteq\mathcal U:A\text{ is 2-linked},\ |[A]|=a,\ |N_H(A)|=g\},
 \qquad t=\rho dg-da.
\]
Here, neighborhoods, closures, and 2-linkedness are taken with respect
to $H$ unless otherwise stated. For the rest of the paper, we omit
the subscript $H$ from all notation whenever the underlying graph is clear.

\begin{lemma}\label{lem:cert}
For all sufficiently large $k$ and every pair of feasible integers
$a,g\ge1$, there exist a family
$\cW(a,g)\subseteq 2^{\mathcal U}\times 2^{\mathcal L}$
and a map $f:\cG(a,g)\to\cW(a,g)$ such that, for every
$A\in\cG(a,g)$, the pair $(S,F)=f(A)$ satisfies
\begin{align}
 [A]\subseteq S,&\qquad F\subseteq N(A),\label{eq:contain}\\
 d_F(u)\ge d-\frac{d}{\log d}\quad(u\in S),&
 \qquad d_{\mathcal U\setminus S}(v)\ge\rho d-\frac{d}{\log d}
 \quad(v\in\mathcal L\setminus F),\label{eq:degree}\\
 |S|&\le\rho|F|+O\left(\frac{t}{d\log d}\right).\label{eq:struct}
\end{align}
Moreover,
\begin{equation}\label{eq:certcount}
 \log|\cW(a,g)|
 \le O\left(k+\frac{(g+t)(\log d)^2}{d^2}\right).
\end{equation}
\end{lemma}

\begin{proof}
Apply Lemma~\ref{thm:BGLW-container} to $H$ with
\[
 q=d,\qquad s=\rho d,\qquad
 \varphi=d/2,\qquad \psi=d/\log d.
\]
Then~\eqref{eq:contain}--\eqref{eq:struct} follow immediately from
Lemma~\ref{thm:BGLW-container}
(i)--(iv).

It remains to verify~\eqref{eq:certcount}.
Every set $K$ in the definition of $m_\varphi$ satisfies
$\varphi<|K|\le\rho d\le k^3$, so
Proposition~\ref{prop:isoperimetry}(i) gives
$m_\varphi\ge d\varphi/\binom{r+4}{r}$.
Take $C=3\binom{r+4}{r}$. Then $Cm_\varphi/(\varphi d)\ge3$,
and the condition $C\ln d/(\varphi d)<1$ holds for
sufficiently large $k$.
Then the exponential factor in~\eqref{eq:certgeneral} is
\[
\exp\left(O\left(\frac{g(\log d)^2}{d^2}\right)+
 O\left(\frac{g\log d}{d^3}\right)+
 O\left(\frac{t(\log d)^2}{d^2}\right)\right) \le \exp\left(O\left(\frac{(g+t)(\log d)^2}{d^2}\right)\right).
\]
Moreover, the product of the three binomial sums in~\eqref{eq:certgeneral}
is at most
\begin{align*}
 &\binom{O(g\log d/d)}{\le O(t\log d/d^2)}
 \binom{O(gd)}{\le O(t\log d/d^2)}
 \binom{O(gd^2)}{\le O(t\log d/d^2)} \le \binom{O(gd^2)}{\le O(t\log d/d^2)}^3\\
\le\quad &
 \left(O\left(\frac{gd^4}{t\log d}\right)\right)^{O(t\log d/d^2)}\le \exp\left(O\left(\frac{t\log d}{d^2}
 \left(\log d+\frac gt\right)\right)\right)\le \exp\left(O\left(\frac{(g+t)(\log d)^2}{d^2}\right)\right)
\end{align*}
for $t>0$, where the second inequality follows from
Proposition~\ref{lem:binomial-estimates}, using
$t\le\rho dg=O(gd)$, and the third uses $\ln(g/t)\le g/t$.
When $t=0$, all three binomial sums equal $1$.

Finally, we have
$|\mathcal L|=M=\exp(O(k))$.
Combining these estimates gives
\[
 |\cW(a,g)|
 \le \exp\left(O\left(k+\frac{(g+t)(\log d)^2}{d^2}\right)\right),
\]
which proves~\eqref{eq:certcount}.
\end{proof}

We next use these containers to establish a counting lemma for large 2-linked sets with sufficient expansion.
The proof follows the standard reconstruction argument used in Sapozhenko's container method; see, e.g.,~\cite[Section~4]{BaloghGarciaLiWagner2024}; we include the details for completeness.
\begin{lemma}\label{lem:reconstruct}
There is an absolute constant $c_0>0$ such that, for all sufficiently large $k$, if $g>k^{12}$ and
$g-a\ge(8r/k)\,g$, then
\[
 |\cG(a,g)|\le2^{g-c_0(g-a)}.
\]
\end{lemma}
\begin{proof}
By the expansion assumption $(8r/k)g\le(g-a)$, we have
\[
 t=d(\rho g - a)=d\left(g-a+\frac rk g\right)=O(d(g-a)).
\]
Then Lemma~\ref{lem:cert} gives
\[
 \log|\cW(a,g)|
 =O\left(k+\frac{g(\log d)^2}{d^2}
          +\frac{(g-a)(\log d)^2}{d}\right)
 =o(g-a),
\]
where the last equality uses $g>k^{12}$,
$g-a\ge8rg/k$, and $d\ge k$.
Thus it suffices to show that, for each $(S,F)\in\cW(a,g)$,
\begin{equation}\label{est:reconstruct}
\text{there are at most $2^{g-\Omega(g-a)}$ sets $A\in\cG(a,g)$ satisfying
$f(A)=(S,F)$}.
\end{equation}

Fix an arbitrary $(S,F)\in\cW(a,g)$ and a sufficiently small
absolute constant $\varepsilon>0$.
If $|S|\le g-\varepsilon(g-a)$, then by~\eqref{eq:contain}, every set $A\in\cG(a,g)$
with $f(A)=(S,F)$ is contained in $S$, so there are at most
$2^{g-\varepsilon(g-a)}$ such sets.
We may therefore assume that \begin{equation}\label{eq:Sbound}
|S|>g-\varepsilon(g-a).
\end{equation}

Fix an arbitrary $A_0\in\cG(a,g)$ with $f(A_0)=(S,F)$, and let
$A^*=[A_0]$ and $N^*=N(A_0)$. The particular choice of $A_0$
does not matter; the pair $(A^*,N^*)$ serves as a fixed anchor,
allowing us to split $N(A)$ into two pieces and handle them differently.
To count the sets $A$ with $f(A)=(S,F)$, we first reconstruct
their neighborhoods $N(A)$, which uniquely determine their closures $[A]$; each closure then allows 
\begin{equation}\label{est:reconstruct1}
\text{at most $2^a$ choices for $A$.}
\end{equation}
To this end, we specify $N(A)$ in two parts:
\[
 N(A)\cap N^*
 \qquad\text{and}\qquad
 N(A)\setminus N^*.
\]
The former is easy to count: since $F\subseteq N(A)\cap N^*$,
the intersection is determined by
$(N(A)\cap N^*)\setminus F\subseteq N^*\setminus F$.
As $|N^*\setminus F|=g-|F|$, 
\begin{equation}\label{est:reconstruct2}
\text{there are at most
$2^{g-|F|}$ choices for $N(A)\cap N^*$}.
\end{equation}

\begin{figure}[ht]
\centering
\begin{tikzpicture}[
  font=\small,
  every node/.style={align=center}
]
\draw (0,2.7) ellipse (3.7 and .85);
\node at (0,3.8) {$S$};

\draw (-1.65,2.7) ellipse (1.3 and .5);
\node at (-1.65,2.7) {$A^*=[A_0]$};
\node at (1.25,3.12) {$S\setminus A^*$};

\node[
  draw,ellipse,
  minimum width=1cm,
  minimum height=.5cm
] (Z) at (1.25,2.55) {$Z_A$};

\node at (4.05,2.7) {$\mathcal U$};

\fill[gray!12] (1.2,0) ellipse (2.3 and .95);

\begin{scope}
  \clip (-1.2,0) ellipse (2.3 and .95);
  \fill[gray!28] (1.2,0) ellipse (2.3 and .95);
\end{scope}

\draw (-1.2,0) ellipse (2.3 and .95);
\draw (1.2,0) ellipse (2.3 and .95);

\node at (-2.1,.05) {$N^*\setminus N(A)$};
\node (outside) at (2.1,.05) {$N(A)\setminus N^*$};
\node[font=\scriptsize] at (0,.32)
  {$N(A)\cap N^*$};

\node[
  draw,ellipse,
  minimum width=.85cm,
  minimum height=.4cm
] at (0,-.22) {$F$};

\node at (-1.8,-1.2) {$N^*=N(A^*)$};
\node at (1.8,-1.2) {$N(A)$};
\node at (4.05,0) {$\mathcal L$};

\draw[->,thick]
  (Z.south east)--(outside.north)
  node[pos=.4,right] {covers};

\end{tikzpicture}
\caption{The reconstruction of $N(A)$.}
\label{fig:reconstruction}
\end{figure}
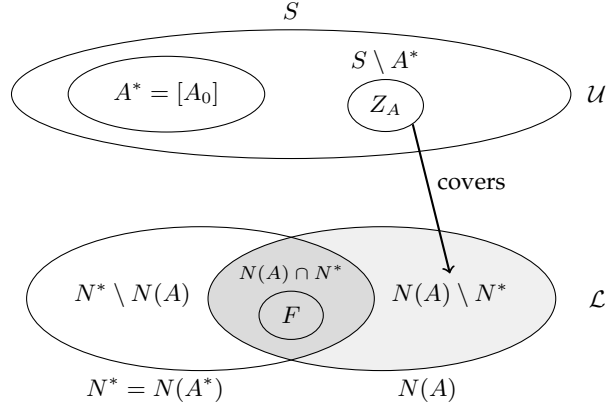

Next, we bound the number of choices for $N(A)\setminus N^*$.
Note that every vertex of $N(A)\setminus N^*$ must have a neighbor in $A\setminus A^*$.
Then for each $A\in\cG(a,g)$ with $f(A)=(S,F)$,
there exists an inclusion-minimal set
$Z_A\subseteq A\setminus A^*$ that covers $N(A)\setminus N^*$, meaning that
\[
 N(Z_A)\setminus N^*=N(A)\setminus N^*.
\]
Fix one such $Z_A$ for each $A$. Once $Z_A$ is chosen, we can uniquely recover
$N(A)\setminus N^*$ by taking $N(Z_A)\setminus N^*$.
By minimality, each vertex of $Z_A$ has a neighbor
in $N(A)\setminus N^*$ adjacent to no other vertex of $Z_A$.
Then we have
\(
 |Z_A|\le |N(A)\setminus N^*|\le g-|F|.
\)
Moreover, $Z_A\subseteq S\setminus A^*$, a set of size $|S|-a$.
Thus,
\begin{equation}\label{est:reconstruct3}
\text{there are at most $\binom{|S|-a}{\le g-|F|}$ choices for $Z_A$, and hence for $N(A)\setminus N^*$.}
\end{equation}

Combining~\eqref{est:reconstruct1}, ~\eqref{est:reconstruct2}, and ~\eqref{est:reconstruct3}, we obtain
\[
 \bigl|\{A\in\cG(a,g):f(A)=(S,F)\}\bigr|
 \le 2^{a}\cdot 2^{g-|F|}\binom{|S|-a}{\le g-|F|}.
\]
It remains to estimate this bound. 
By Lemma~\ref{lem:cert} and $t=O(d(g-a))$,
\[
 |S|\le\rho|F|+O\left(\frac{g-a}{\log d}\right)
 =\rho|F|+o(g-a).
\]
Together with $|F|\le g$ and the expansion assumption
$(8r/k)g\le g-a$, this inequality gives
\[
|S| -a \le \rho g - a +o(g-a)=\frac rkg + (1+o(1))(g-a) \le \left(\frac 98 + o(1)\right)(g-a),
\]
and
\[
g-|F| \le \rho(g-|F|)\le \rho g -|S| + o(g-a) \le \frac rk g+(\eps + o(1))(g-a)\le \left(\frac18+\eps+o(1)\right)(g-a),
\]
where the third inequality uses~\eqref{eq:Sbound}.
Note that~\eqref{eq:Sbound} also gives $|S|-a>(1-\varepsilon)(g-a)$, so
\[(g-|F|)/(|S|-a)\le1/6\] for sufficiently small $\varepsilon$ and sufficiently large $k$.

By Proposition~\ref{lem:binomial-estimates} and
$H_2(1/6)<2/3$,
\[
 \log\binom{|S|-a}{\le g-|F|}
 \le H_2(1/6)(|S|-a)
 \le \frac23(|S|-a).
\]
Combining all the estimates, we obtain
\[
\log \bigl|\{A\in\cG(a,g):f(A)=(S,F)\}\bigr| \le a+(g-|F|)+\frac23(|S|-a)
 \le a+\left(\frac 78 + \eps+o(1)\right)(g-a)
 =g-\Omega(g-a),
\]
which establishes~\eqref{est:reconstruct} and thereby completes the proof.
\end{proof}

For small 2-linked sets, we do not use the container lemma.
Instead, the isoperimetric bound in Proposition~\ref{prop:isoperimetry} gives stronger expansion for them, so a more elementary counting argument for connected subsets suffices to obtain the following weighted estimates.

\begin{lemma}\label{lem:small}
Let
\[
 \cG_{\mathrm s}
 :=\{A\subseteq\mathcal U:A\ne\varnothing,\ A\text{ is 2-linked},
       |N(A)|\le k^{12}\}.
\]
For sufficiently large $k$,
\begin{align}
 \sum_{A\in\cG_{\mathrm s}}2^{-|N(A)|}
 &=O(M2^{-d/200}),\label{eq:poly1}\\
 \sum_{A\in\cG_{\mathrm s}}2^{-\lambda(|N(A)|-|A|)}
 &\le4M2^{-\lambda d/400},\label{eq:poly2}
\end{align}
where the second estimate holds uniformly for
$C\log k/d\le\lambda\le1$, for a sufficiently large absolute $C$.
\end{lemma}

\begin{proof}
For every $A\in\cG_{\mathrm s}$, biregularity gives
$|A|\le\rho|N(A)|\le2k^{12}$.
Proposition~\ref{prop:isoperimetry}(ii) therefore gives
\[
 |N(A)|\ge\frac{d|A|}{200},
 \qquad
 |N(A)|-|A|\ge\frac{d|A|}{400}
\]
for sufficiently large $k$.
Since $H^2[\mathcal U]$ has maximum degree at most $\rho d^2$,
Proposition~\ref{prop:connected-sets} bounds the number of
2-linked sets of size $s$ by $N(e\rho d^2)^{s-1}$.
Consequently,
\[
 \sum_{A\in\cG_{\mathrm s}}2^{-|N(A)|}
 \le N\sum_{s\ge1}(e\rho d^2)^{s-1}2^{-ds/200}
 =\frac{N2^{-d/200}}{1-e\rho d^2 2^{-d/200}}
 =O(M2^{-d/200}),
\]
proving~\eqref{eq:poly1}.
Similarly,
\[
 \sum_{A\in\cG_{\mathrm s}}
 2^{-\lambda(|N(A)|-|A|)}
 \le N\sum_{s\ge1}(e\rho d^2)^{s-1}2^{-\lambda ds/400}
 =\frac{N2^{-\lambda d/400}}
 {1-e\rho d^2 2^{-\lambda d/400}}
 \le4M2^{-\lambda d/400}.
\]
The last inequality holds for sufficiently large $k$, since $N\le2M$ and $\lambda d\ge C\log k$ gives
$e\rho d^2 2^{-\lambda d/400}=o(1)$ for sufficiently large $C$.
This proves~\eqref{eq:poly2}.
\end{proof}

\section{Proof of Lemma~\ref{thm:global}: global phase selection}\label{sec:global}
In this section, we prove Lemma~\ref{thm:global}, which bounds the
number of intersecting families $\cF\subseteq\binom{[n]}k$ satisfying
\[
 \dv(\cF)=\min_{x\in[n]}|\cF\setminus\cS_x|>\xi_1M.
\]
We will first bound the number of independent sets (that is, intersecting families) with large diversity in ``sparse" Kneser containers, and then combine this bound with the estimates from Section~\ref{sec:containers} to complete the proof.

For simplicity, throughout this section we write $G=\KG(n,k)$ for the Kneser graph.

\subsection{Counting in sparse Kneser containers}\label{sec:Kcontainer}
We first present the following supersaturation bound, which follows immediately from the degree and least eigenvalue of the Kneser graph. We use the formulation appearing in Yang~\cite{Yang2026}; see also Lov\'asz~\cite{Lovasz1979} for the spectrum of the Kneser graph.

\begin{lemma}[Supersaturation]\label{prop:kneser-super}
Let $n\ge2k$. For every $W\subseteq V(G)$ with $|W|\ge\binom{n-1}{k-1}$,
\begin{equation*}
 2e_G(W)\ge
 \frac{\binom{n-k}{k}+\binom{n-k-1}{k-1}}{\binom nk}
 |W|\left(|W|-\binom{n-1}{k-1}\right).
\end{equation*}
\end{lemma}

\begin{proof}
We include a short
proof for completeness.
Recall that Lovász~\cite{Lovasz1979} showed that the eigenvalues of $G$ are
$(-1)^j\binom{n-k-j}{k-j}$ for $0\leq j\leq k$. In particular,
$G$ is $\binom{n-k}{k}$-regular and has least eigenvalue
$-\binom{n-k-1}{k-1}$.

Let $A$ be the adjacency matrix of $G$, $\1_W$ be the indicator vector of $W$, and $\1$ be the all-ones vector.
Note that we can write
$\1_W=\frac{|W|}{\binom nk}\1+f$, where $f\perp\1$, and that
$\|f\|^2=|W|-\frac{|W|^2}{\binom nk}$. Then it follows that
\[
 \begin{aligned}
 2e_G(W)
 &=\1_W^{T}A\1_W
=\frac{|W|^2}{\binom{n}{k}^2}\1^{T}A\1 + f^{T}Af\ge\frac{\binom{n-k}{k}}{\binom nk}|W|^2-\binom{n-k-1}{k-1}\|f\|^2\\
 &\geq
 \frac{\binom{n-k}{k}+\binom{n-k-1}{k-1}}{\binom nk}|W|^2
 -\binom{n-k-1}{k-1}|W|.
 \end{aligned}
\]
Since
\(
 \left(\binom{n-k}{k}+\binom{n-k-1}{k-1}\right)
 \binom{n-1}{k-1}
 =
 \binom{n-k-1}{k-1}\binom nk,
\)
the desired inequality follows.
\end{proof}

With the notation introduced in Section~\ref{sec:notation},
Lemma~\ref{prop:kneser-super} gives
\[
 2e_G(W)\ge\frac dM|W|(|W|-M)
 \qquad\text{whenever }|W|\ge M.
\]
Consequently, for every $W\subseteq V(G)$,
\begin{equation}\label{eq:size-e}
 |W|\le M+\frac{2e_G(W)}d.
\end{equation}

We also state the following special case of the Das--Tran removal
lemma~\cite[Theorem~1.2]{DasTran2016} for set families, obtained by setting
$\ell=1$ in their notation.
\begin{lemma}[Das--Tran~\cite{DasTran2016}]\label{prop:DT}
There is an absolute constant $C_{\mathrm{DT}}>1$ such that the following holds. Let $n>2k$, and let $\cF\subseteq\binom{[n]}k$ satisfy
\[
 |\cF|=(1-\alpha)M,
\]
with at most $\beta Md$ disjoint pairs, where $\alpha\in \mathbb{R}$, $\beta\ge0$, and
\begin{equation}\label{eq:DT-hypothesis}
 \max\{2|\alpha|,\beta\}
 \le\frac{n-2k}{(20C_{\mathrm{DT}})^2n}.
\end{equation}
Then there is a full star $\cS_x$, with $x\in[n]$, such that
\[
 |\cF\mathbin\triangle\cS_x|
 \le C_{\mathrm{DT}}(\alpha+2\beta)\frac{n}{n-2k}M.
\]
\end{lemma}

Using the supersaturation bound and the removal lemma, we show that every ``sparse'' container $W$ in the Kneser graph contains only a negligible number of independent sets with large diversity.
\begin{lemma}\label{lem:sparse}
For all sufficiently large $k$, for every $0<\zeta\le1$ and every $W\subseteq V(G)$ with \[\Delta(G[W])\le2d/\log d,\] the number of independent sets $\cI\subseteq W$ satisfying $\dv(\cI)>\zeta M$ is at most $2^{M-\Omega(\zeta M/k)}$.
\end{lemma}
\begin{proof}
Let $c>0$ be a small absolute constant to be chosen later.
If $|W|\le M-c\zeta M/k$, simply counting all subsets of $W$
gives the desired bound. Hence we may assume
$|W|>M-c\zeta M/k$.

We may further assume that $W$ contains at least one independent set $\cI$
with $\dv(\cI)>\zeta M$; otherwise, there is
nothing to count.
We claim that 
\begin{equation}\label{ineq:eW}
    e_G(W)\ge c\zeta Md/k.
\end{equation}
Suppose not, i.e., $e_G(W)<c\zeta Md/k$, so $W$
contains fewer than $c\zeta Md/k$ disjoint pairs.
Let $\alpha=1-|W|/M$ and $\beta=c\zeta/k$.
Our lower bound on $|W|$ gives $\alpha\le c\zeta/k$.
Moreover,~\eqref{eq:size-e} gives
$\alpha\ge-2e_G(W)/(Md)>-2\beta=-2c\zeta/k$,
so $|\alpha|\le2c\zeta/k$.
Let $c:=1/(10^4C_{\mathrm{DT}}^2)$.
Then condition~\eqref{eq:DT-hypothesis} holds, and thus by
Lemma~\ref{prop:DT}, there exists a full star $\cS_x$ with
\[
 |W\mathbin\triangle\cS_x|
 \le C_{\mathrm{DT}}(\alpha+2\beta)\frac nr M
 \le9C_{\mathrm{DT}}c\zeta M<\zeta M.
\]
Thus every $\cI\subseteq W$ satisfies
$\dv(\cI)\le|W\setminus\cS_x|<\zeta M$,
a contradiction.

Since $\Delta(G[W])\le2d/\log d$, greedily choosing edges
and deleting all edges incident to their endpoints
yields a matching $\mathcal{M}$
of size $|\mathcal{M}|\ge e_G(W)\log d/(4d)$.
Note that an independent set contains at most one endpoint of each
matching edge. 
It follows that
\[
\begin{aligned}
 \log\#\{\cI\subseteq W:\cI\text{ is independent}\}
 &\le \log\left(3^{|\mathcal M|}2^{|W|-2|\mathcal M|}\right)
 \le M-\left(\frac{\log(4/3)\log d}{4}-2\right)
          \frac{e_G(W)}d\\
 &\le M - \frac{e_G(W)}d \le M-c\zeta M/k,
\end{aligned}
\]
where the second inequality uses~\eqref{eq:size-e} and the
bound on $|\mathcal M|$, the third holds for sufficiently
large $k$, and the last follows from~\eqref{ineq:eW}.
This completes the proof.

\end{proof}

To conclude this subsection, we give some intuition for why we count independent sets in the sparse containers considered in Lemma~\ref{lem:sparse}.
Fix $x\in[n]$, and let $H$ be the corresponding bipartite inclusion graph from Section~\ref{sec:containers}, with parts $\mathcal U$ and $\mathcal L$.
For a certificate $(S,F)$ produced by Lemma~\ref{lem:cert}, define
\[
 W(S,F):=\{[n]\setminus(u\cup\{x\}):u\in S\}
 \cup \{\{x\}\cup v:v\in\mathcal L\setminus F\}.
\]
Note that $W(S,F)\subseteq V(G)$.
Under this correspondence, edges of $H$ are precisely the
Kneser edges between $\cS_x$ and $V(G)\setminus\cS_x$.
Thus~\eqref{eq:degree} implies that each vertex in
$W(S,F)\cap\cS_x$ has at most $d/\log d$ neighbors in
$W(S,F)\setminus\cS_x$, and vice versa.
Moreover, $W(S,F)\cap\cS_x$ is independent, since all its
members contain $x$.
The degree within $W(S,F)\setminus\cS_x$ is at most $rd/k$,
because $G[V(G)\setminus\cS_x]$ is isomorphic to $\KG(n-1,k)$, which
is regular of degree
\begin{equation}\label{eq:outside-degree}
 \binom{n-1-k}{k}=\binom{k+r-1}{r-1}=\frac{rd}{k}.
\end{equation}
Consequently, for sufficiently large $k$,
\begin{equation}\label{eq:certificate-degree}
 \Delta(G[W(S,F)])\le\frac{d}{\log d}+\frac{rd}{k}
 \le\frac{2d}{\log d}.
\end{equation}
Hence every such $W(S,F)$ satisfies the degree hypothesis of Lemma~\ref{lem:sparse}.

\subsection{Proof of Lemma~\ref{thm:global}}
\begin{proof}[Proof of Lemma~\ref{thm:global}]
Let $\xi_1\ge C_1(\log k)^2/d$, where $C_1$ is a sufficiently large absolute constant.
We may assume $\xi_1\le1$, since otherwise the
Erd\H{o}s--Ko--Rado theorem gives $|\cF|\le M<\xi_1M$,
so there are no families under consideration.
Fix a center $x\in[n]$, and let $H$ be the corresponding
bipartite inclusion graph. Throughout the proof, neighborhoods, closures, and 2-linkedness are taken with respect to $H$, unless otherwise stated.

By Lemma~\ref{lem:encoding}, each family $\cF$ is determined by
$A_x(\cF)$ and $B_x(\cF)$; moreover, once $A_x(\cF)$ is fixed, there are exactly
$2^{M-|N(A_x(\cF))|}$ intersecting families with this $A_x(\cF)$.
Thus, the core of the proof is to control the possible choices for $A_x(\cF)$.
To this end, we view each candidate for $A_x(\cF)$ as a union
of its 2-linked components, whose neighborhoods are disjoint.
We say a component $A$ is \emph{small} if $|N(A)|\le k^{12}$,
and \emph{large} otherwise.
We consider the small and large components separately.
Briefly, all the small components will be handled using
Lemma~\ref{lem:small}. For the large components, instead of
dealing with them individually, we join them into one
2-linked set. If this set has sufficient expansion, we use
Lemma~\ref{lem:reconstruct}; otherwise, we use
Lemma~\ref{lem:cert} together with Lemma~\ref{lem:sparse}.

For now, fix an arbitrary intersecting family $\cF$ with $\dv(\cF)>\xi_1M$ and let $A_1,\ldots,A_m$ be the 2-linked components of $A_x(\cF)$.
We first establish properties of these components that
will be used in the subsequent counting argument.
Let
\[
 A_{\mathrm s}
 =\bigcup_{A_i\text{ small}}A_i,
 \qquad
 A_{\mathrm{lg}}
 =\bigcup_{A_i\text{ large}}A_i.
\]
For each small component $A_i$, the biregularity of $H$ gives $|A_i|\le\rho|N(A_i)|\le2k^{12}$.
Then by Proposition~\ref{prop:isoperimetry}(ii), we have
$|A_i|\le200|N(A_i)|/d$, and thus
\[
 |A_{\mathrm s}|
 \le \frac{200}{d}\sum_{A_i\text{ small}}|N(A_i)|
 \le \frac{200|\mathcal L|}{d}
 =\frac{200M}{d},
\]
where the second inequality uses the fact that different components have disjoint neighborhoods.
Together with $|A_x(\cF)|\ge\dv(\cF)>\xi_1M$ and
$\xi_1d\ge C_1(\log k)^2$, this shows that for sufficiently large $k$
\begin{equation}\label{eq:Algbound}
     |A_{\mathrm{lg}}|\ge\frac34\xi_1M.
\end{equation}

We next join the large components into one 2-linked set.

\begin{claim}\label{claim:largecomp}
There exists a set
$T\subseteq\mathcal U\setminus A_{\mathrm{lg}}$ with
\[
 |T|\le\frac{|N(A_{\mathrm{lg}})|}{k^{11}}
\]
such that $A_{\mathrm{lg}}\cup T$ is 2-linked.
\end{claim}

\noindent\emph{Proof of the claim.}
There are fewer than $|N(A_{\mathrm{lg}})|/k^{12}$
large components, since their neighborhoods are disjoint
and each has size greater than $k^{12}$.
Note that any two vertices $u,v\in\mathcal U$ can be joined by a path of length at most $k$ in $H^2[\mathcal U]$ (by replacing one element at a time), since $|u\setminus v|\le k$.
Choose a representative vertex from each large component and join
each to the first by such a path. Let $T$ be the set of
vertices on these paths outside $A_{\mathrm{lg}}$.
Then $A_{\mathrm{lg}}\cup T$ is 2-linked and
$|T|\le |N(A_{\mathrm{lg}})|/k^{11}$, as required.

For each possible $A_{\mathrm{lg}}$, fix one set $T$
satisfying the claim. Thus each family $\cF$ determines
\[
 \widehat A:=A_{\mathrm{lg}}\cup T,\qquad
 a:=|[\widehat A]|,\qquad g:=|N(\widehat A)|, \qquad g_{\mathrm{lg}}:=|N(A_{\mathrm{lg}})|.
\]
The set $N(\widehat A)\setminus N(A_{\mathrm{lg}})$ consists of the new neighbors introduced when joining the large components. Since $d\le k^2$,
\begin{equation}\label{eq:joining}
 |N(\widehat A)\setminus N(A_{\mathrm{lg}})|=g-g_{\mathrm{lg}}
 \le d|T|\le\frac{g_{\mathrm{lg}}}{k^9}
 \le\frac{M}{k^9}.
\end{equation}
Also, by biregularity,~\eqref{eq:Algbound}, and $\rho\le2$,
\[
 g\ge g_{\mathrm{lg}}
 \ge\frac{|A_{\mathrm{lg}}|}{\rho}
 \ge\frac{3\xi_1M}{8}>k^{12},
\]
for sufficiently large $k$.

Using these properties, we now count the intersecting families $\cF$
with $\dv(\cF)>\xi_1M$ in two cases, according to the
expansion of their associated sets $\widehat A$.

\medskip
\noindent
\textbf{Case 1: $g-a\ge8rg/k$.}

Fix values of $a,g$ satisfying this inequality. We first count the possible choices for $A_{\mathrm{lg}}$, then sum over all possible choices for $A_{\mathrm s}$, and count the families $\cF$ with $A_x(\cF)=A_{\mathrm{lg}}\cup A_{\mathrm s}$.

By Lemma~\ref{lem:reconstruct}, there are at most $2^{g-\Omega(g-a)}$ choices for $\widehat A$.
Since $A_{\mathrm{lg}}=\widehat A\setminus T$ and $|T|\le g/k^{11}$, each $\widehat A$ allows at most
\[
 \sum_{j\le g/k^{11}}\binom Nj
 =2^{O(g/k^{10})}
\]
choices for $T$, and hence for $A_{\mathrm{lg}}$, by Proposition~\ref{lem:binomial-estimates} and $\log N=O(k)$.

Let
\[
 \mathcal A_{\mathrm s}
 :=\{A\subseteq\mathcal U:
       |N(C)|\le k^{12}
       \text{ for every 2-linked component }C\text{ of }A\},
\]
including the empty set.
Note that $A_{\mathrm s}\in\mathcal A_{\mathrm s}$ uniquely determines its small components as its 2-linked components.
Summing over all these choices of $a$, $g$, $\widehat A$, $A_{\mathrm{lg}}$, and $A_{\mathrm s}$, and using Lemma~\ref{lem:encoding}, we conclude that the total number of intersecting families in Case~1 is at most
\[
 \sum_{a,g}
 \sum_{\widehat A\in\cG(a,g)}
 \sum_{A_{\mathrm{lg}}}
 \sum_{A_{\mathrm s}\in\mathcal A_{\mathrm s}}
 2^{M-g_{\mathrm{lg}}-|N(A_{\mathrm s})|}
 \le \sum_{a,g}
 \sum_{\widehat A\in\cG(a,g)}
 \sum_{A_{\mathrm{lg}}}2^{M-g + g/k^9}
 \sum_{A_{\mathrm s}\in\mathcal A_{\mathrm s}}
 2^{-|N(A_{\mathrm s})|},
\]
where the inequality follows from~\eqref{eq:joining}.
Here the sum for $a, g$ ranges over $0\le a\le N$, $3\xi_1M/8\le g\le M$, and $g-a\ge8rg/k$.

We next estimate the sum over $A_{\mathrm s}$.
Each $A_{\mathrm s}\in\mathcal A_{\mathrm s}$ uniquely
determines a collection $\mathcal C_{\mathrm s}$ of members of $\cG_{\mathrm s}$
with $|N(A_{\mathrm s})|=\sum_{A\in \mathcal{C}_{\mathrm s}} |N(A)|$. Thus, by
Lemma~\ref{lem:small},
\[
 \begin{aligned}
 \sum_{A_{\mathrm s}\in\mathcal A_{\mathrm s}}
 2^{-|N(A_{\mathrm s})|}
 &\le
 \sum_{\mathcal C_{\mathrm s}\subseteq\cG_{\mathrm s}}
 \prod_{A\in\mathcal C_{\mathrm s}}2^{-|N(A)|}=\prod_{A\in\cG_{\mathrm s}}
 \left(1+2^{-|N(A)|}\right)\\
 &\le
 \exp\left(\sum_{A\in\cG_{\mathrm s}}2^{-|N(A)|}\right)
 =\exp\left(O(M/2^{d/200})\right),
 \end{aligned}
\]
where the first inequality allows all subcollections of
$\cG_{\mathrm s}$, without requiring their neighborhoods
to be disjoint.

Finally, combining all the bounds above, we conclude that
the total number of families in Case~1 is at most
\begin{equation}\label{inq:case1bound}
 \begin{aligned}
 &\sum_{a,g}
 2^{g-\Omega(g-a)}
2^{O(g/k^{10})}
2^{M-g+g/k^9}
 \exp\left(O(M/2^{d/200})\right)=
 \sum_{a,g}
 2^{M-\Omega(g-a)+O(g/k^9)+O(M/2^{d/200})}\\
\le\quad  &
 (N+1)(M+1)\,2^{M-\Omega(\xi_1M/k)}
 =2^{M-\Omega(\xi_1M/k)}.
 \end{aligned}
\end{equation}
Here $g-a\ge8rg/k$ and $g\ge3\xi_1M/8$ imply
$g-a=\Omega(\xi_1M/k)$.
Together with the lower bound on $\xi_1$, these ensure that the saving $\Omega(g-a)$ absorbs all error terms.
The factor $(N+1)(M+1)$ counts the possible pairs $(a,g)$ and is also absorbed into this saving, since
$\log((N+1)(M+1))=O(k)=o(\xi_1M/k)$.

\medskip
\noindent
\textbf{Case 2: $g-a<8rg/k$.}

Under this expansion condition, Lemma~\ref{lem:reconstruct} does not apply. 
Instead, we first fix a certificate $(S,F)$ for $\widehat A$; the number of possible certificates is small by Lemma~\ref{lem:cert}.
We then fix a choice for $A_{\mathrm s}$, and use $(S,F)$ and $A_{\mathrm s}$ to construct a sparse Kneser container, to which we can apply Lemma~\ref{lem:sparse}.

For each feasible pair $(a,g)$ in this case, we have
\begin{equation}\label{inq:tbound}
     0\le t=d(\rho g-a)=d((r/k)g+g-a)<9drg/k.
\end{equation}
Thus Lemma~\ref{lem:cert} gives
\[
     \log|\cW(a,g)|
 \le O\left(k+\frac{(g+t)(\log d)^2}{d^2}\right)
 =O\left(\frac{M(\log k)^2}{kd}\right),
\]
where we used $g\le M$ and $k\le d\le k^2$.
Summing over all feasible pairs
$a,g$ as in Case~1, we have
\begin{equation}\label{eq:lowcerttotal}
 \sum_{a,g}|\cW(a,g)|
 \le 2^{O(M(\log k)^2/(kd))},
\end{equation}
where the sum ranges over $0\le a\le N$, $3\xi_1M/8\le g\le M$, and $g-a<8rg/k$.

Now fix a feasible pair $(a,g)$ in this range and a certificate
$(S,F)\in\cW(a,g)$, and let
\[
 \mathfrak F(S,F)
 :=\{\cF:\cF\text{ is intersecting},\
          \dv(\cF)>\xi_1M,\
          f(\widehat A)=(S,F)\},
\]
where $\widehat A$ is the 2-linked set obtained by joining the large components of $A_x(\cF)$ as in Claim~\ref{claim:largecomp}.
The certificate set $F\subseteq N(\widehat A)$ may intersect $B_x(\cF)$, but only within the additional neighborhood $N(\widehat A)\setminus N(A_{\mathrm{lg}})$, since $N(A_{\mathrm{lg}})\cap B_x(\cF)=\varnothing$.
For each $\cF\in\mathfrak F(S,F)$, define
\[
 \cF_0:=\cF\setminus\{\{x\}\cup v:v\in B_x(\cF)\cap F\}.
\]
Thus, $\cF$ is fully determined by $B_x(\cF)\cap F$ and $\cF_0$.
By~\eqref{eq:joining}, we have
\[
|B_x(\cF)\cap F|\le |N(\widehat A)\setminus N(A_{\mathrm{lg}})|\le M/k^9,
\] and thus the number of choices for $B_x(\cF)\cap F$ is at most
\begin{equation}\label{eq:Dcost}
 \sum_{j\le M/k^9}\binom Mj=2^{O(M/k^8)},
\end{equation}
by Proposition~\ref{lem:binomial-estimates} and $\log M=O(k)$.
Moreover
\begin{equation}\label{inq:div}
 \dv(\cF_0)\ge\dv(\cF)-|B_x(\cF)\cap F|>\xi_1M/2
\end{equation}
for sufficiently large $k$, since $\xi_1\ge C_1(\log k)^2/d$ and $d\le k^2$.

To count the possible families $\cF_0$, we fix a choice
for $A_{\mathrm s}\in\mathcal A_{\mathrm s}$ and let
\[
 \mathfrak F_0(S,F,A_{\mathrm s})
 :=\{\cF_0:\cF\in\mathfrak F(S,F)
       \text{ with the small components }A_{\mathrm s}\}.
\]

\begin{claim}\label{claim:case2count}
There is an absolute constant $\eta>0$ such that
\[
 \log|\mathfrak F_0(S,F,A_{\mathrm s})|
 \le M-\eta\xi_1\left(M/k+|N(A_{\mathrm s})|-|A_{\mathrm s}|\right).
\]
\end{claim}

\begin{proof}[Proof of the claim.]
We may assume $\mathfrak F_0(S,F,A_{\mathrm s})\ne\varnothing$.
Define the Kneser container
\[
 W=W(S,F,A_{\mathrm s})
 :=\{[n]\setminus(u\cup\{x\}):u\in S\cup A_{\mathrm s}\}
 \cup
 \{\{x\}\cup v:v\in\mathcal L\setminus(F\cup N(A_{\mathrm s}))\}.
\]
By~\eqref{eq:contain}, we have
$A_{\mathrm{lg}}\cup A_{\mathrm s}\subseteq
\widehat A\cup A_{\mathrm s}\subseteq S\cup A_{\mathrm s}$,
and then $\cF_0\setminus\cS_x=\cF\setminus\cS_x$ lies in the
first part of $W$. Also, $\cF_0\cap\cS_x$ lies in the second
part, since its encodings avoid both $F$ and $N(A_{\mathrm s})$.
Thus, for every $\cF_0\in\mathfrak F_0(S,F,A_{\mathrm s})$,
\[
 \cF_0\subseteq W.
\]

We first apply Lemma~\ref{lem:sparse}.
Since the vertices of $G$ obtained from $A_{\mathrm s}$ have no neighbors in $W\cap\cS_x$, an argument similar to that used to prove~\eqref{eq:certificate-degree} at the end of Section~\ref{sec:Kcontainer} gives
\[
 \Delta(G[W])\le\frac d{\log d}+\frac{rd}{k}
 \le\frac{2d}{\log d}
\]
for sufficiently large $k$.
By~\eqref{inq:div}, every $\cF_0\in\mathfrak F_0(S,F,A_{\mathrm s})$ has diversity greater than $\xi_1M/2$.
Applying Lemma~\ref{lem:sparse} to $W$ with $\zeta=\xi_1/2$ therefore gives
\begin{equation}\label{eq:case2diversity}
 \log|\mathfrak F_0(S,F,A_{\mathrm s})|
 \le M-c_0\xi_1M/k
\end{equation}
for an absolute constant $c_0>0$.

The above bound gives a saving but does not depend on $A_{\mathrm s}$.
To control the sum over all choices of $A_{\mathrm s}$, we also need a bound on $|W|$ in terms of $A_{\mathrm s}$.
By the certificate property~\eqref{eq:struct},
\[
 |S|-|F|
 \le\frac rk|F|+O\left(\frac{t}{d\log d}\right)
 \le\frac{3M}{k},
\]
where the last inequality uses~\eqref{inq:tbound}, $|F|\le g\le M$, and $r\le2$, for sufficiently large $k$.
Since the small and large components have disjoint
neighborhoods and $F\subseteq N(\widehat A)$, we have
\[
 F\cap N(A_{\mathrm s})
 \subseteq N(\widehat A)\setminus N(A_{\mathrm{lg}}).
\]
Together with~\eqref{eq:joining}, this gives
\[
\begin{aligned}
 |W|
 &=|S\cup A_{\mathrm s}|+M-|F\cup N(A_{\mathrm s})|
 \le M+|S|-|F|
       -\bigl(|N(A_{\mathrm s})|-|A_{\mathrm s}|\bigr)
       +|F\cap N(A_{\mathrm s})|\\
 &\le M+\frac{3M}{k} +\frac{M}{k^9}
       -\bigl(|N(A_{\mathrm s})|-|A_{\mathrm s}|\bigr)\le M+\frac{4M}{k}
       -\bigl(|N(A_{\mathrm s})|-|A_{\mathrm s}|\bigr).
\end{aligned}
\]
Simply counting all subsets of $W$, we obtain
\begin{equation}\label{eq:case2size}
 \log|\mathfrak F_0(S,F,A_{\mathrm s})|
 \le M+4M/k-(|N(A_{\mathrm s})|-|A_{\mathrm s}|).
\end{equation}

Let $\lambda=\eta\xi_1$, where $\eta=\min\{c_0/10,\,1/2\}$. Combining~\eqref{eq:case2diversity}
and~\eqref{eq:case2size} using
$\min\{a,b\}\le a^{1-\lambda}b^\lambda$, we obtain
\[
\begin{aligned}
 \log|\mathfrak F_0(S,F,A_{\mathrm s})|
 &\le (1-\lambda)(M-c_0\xi_1M/k) + \lambda(M+4M/k-(|N(A_{\mathrm s})|-|A_{\mathrm s}|))\\
 &\le M-[c_0(1-\lambda)-4\eta]\xi_1M/k
       -\eta\xi_1(|N(A_{\mathrm s})|-|A_{\mathrm s}|)\\
 &\le
 M-\eta\xi_1\left(M/k+|N(A_{\mathrm s})|-|A_{\mathrm s}|\right),
\end{aligned}
\]
where the last inequality uses $\xi_1\le1$ and the choice
of $\eta$. This proves the claim.
\end{proof}

\medskip

Combining the previous discussion and using~\eqref{eq:lowcerttotal}, \eqref{eq:Dcost}, and Claim~\ref{claim:case2count},
we conclude that the total number of families in Case~2
is at most
\[
\begin{aligned}
 & \sum_{a,g}\sum_{(S,F)\in\cW(a,g)}\sum_{\substack{B\subseteq F\\|B|\le M/k^9}}
 \sum_{A_{\mathrm s}\in\mathcal A_{\mathrm s}}
 |\mathfrak F_0(S,F,A_{\mathrm s})|
 \le 2^{O(M(\log k)^2/(kd))}2^{O(M/k^8)}
 \sum_{A_{\mathrm s}\in\mathcal A_{\mathrm s}}
2^{M-\eta\xi_1\left(M/k+|N(A_{\mathrm s})|-|A_{\mathrm s}|\right)}\\
 &\qquad\le 2^{M - \eta\xi_1M/2k}
 \sum_{A_{\mathrm s}\in\mathcal A_{\mathrm s}}
 2^{-\eta\xi_1(|N(A_{\mathrm s})|-|A_{\mathrm s}|)},
\end{aligned}
\]
where the last inequality follows from $\xi_1\ge C_1(\log k)^2/d$ with $C_1$ sufficiently large.
By a calculation similar to that in Case~1, we obtain
\[
 \sum_{A_{\mathrm s}\in\mathcal A_{\mathrm s}}
 2^{-\eta\xi_1(|N(A_{\mathrm s})|-|A_{\mathrm s}|)}
 \le \exp\left(\sum_{A\in\cG_{\mathrm s}}
          2^{-\eta\xi_1(|N(A)|-|A|)}\right)
 \le \exp\left(O\left(M2^{-\eta\xi_1 d/400}\right)\right)
 =2^{o(\xi_1M/k)},
\]
where we use $\eta\xi_1d\ge\eta C_1(\log k)^2$ to apply \eqref{eq:poly2}, and this bound together with $d\le k^2$ to obtain the last equality.

Thus,  the total number of families in Case~2
is at most
\[
2^{M - \Omega(\xi_1M/k)}.
\]
Together with the bound~\eqref{inq:case1bound} in Case~1, this completes the proof of Lemma~\ref{thm:global}.
\end{proof}

\section{Polymer models and the cluster expansion}\label{sec:polymer-background}

This section briefly introduces polymer models and the cluster expansion in graph-theoretic language. 
For more background and applications, see, e.g., Fern\'andez and Procacci~\cite{FernandezProcacci2007}, Jenssen and Perkins~\cite{JenssenPerkins2020}, and Koteck\'y and Preiss~\cite{kotecky1986cluster}; a textbook introduction to the cluster expansion appears in Friedli and Velenik~\cite[Chapter~5]{FriedliVelenik2017}. 

\medskip\noindent\textbf{Abstract polymer model.}
Let $H_{\cP}$ be a graph on a finite set $\cP$, with a loop at each vertex and no multiple edges. The vertices $\gamma\in\cP$ are called \emph{polymers}. Two polymers $\gamma,\gamma'$ are adjacent, written $\gamma\sim\gamma'$, when $\gamma\gamma'$ is an edge of $H_{\cP}$; in particular, each polymer is adjacent to itself. Equip every polymer $\gamma$ with a complex weight $\omega(\gamma)$ and a positive-integer size $|\gamma|$. In our application, polymers are finite vertex sets and $|\gamma|$ is their cardinality.
We refer to this weighted graph $(H_{\cP},\omega)$ as the \emph{polymer model}.

Let $\Omega_{\cP}$ be the collection of independent sets of $H_{\cP}$, including the empty set, where loops are ignored for independence.
We call each such independent set a \emph{compatible configuration}.
The \emph{polymer model partition function} is defined as
\[
\Xi(\cP,\omega):=\sum_{\Lambda\in\Omega_{\cP}}\prod_{\gamma\in\Lambda}\omega(\gamma),
\]
which is essentially the weighted independence polynomial of $(H_{\cP},\omega)$.
Sometimes, by abuse
of notation, we also refer to  $(\cP,\omega)$, $\Xi(\cP,\omega)$ as the polymer model.

\medskip\noindent\textbf{Cluster expansion.}
For an ordered multiset $\Gamma=(\gamma_1,\ldots,\gamma_m)$ of polymers, the \emph{incompatibility graph} $H_{\cP}[\Gamma]$ is the simple graph on $\Gamma$ with the edge set $E=\{\gamma_i\gamma_j:\ \gamma_i\sim\gamma_j \text{ in } H_{\cP}\}$. 
The ordered multiset $\Gamma$ is a \emph{cluster} if it is nonempty and
$H_{\cP}[\Gamma]$ is connected. For example, if $\gamma\sim\gamma'$,
then $\Gamma=(\gamma,\gamma',\gamma)$ is a cluster with
$H_{\cP}[\Gamma]=K_3$.

For a simple graph $G$, define its Ursell function by
\[
 \phi(G):=\frac1{|V(G)|!}\sum_{\substack{F\subseteq G\text{ spanning}\\F\text{ connected}}}(-1)^{e(F)}.
\]
The weight and total size of a cluster are then defined as
\[
 \omega(\Gamma):=\phi(H_{\cP}[\Gamma])\prod_{\gamma\in\Gamma}\omega(\gamma),
 \qquad
 \|\Gamma\|:=\sum_{\gamma\in\Gamma}|\gamma|.
\]
Thus total cluster size counts the sizes of the polymers with multiplicity. 

Let $\cC$ be the set of all clusters. The \emph{cluster expansion} is the formal power series of the logarithm of the partition function $\Xi(\cP,\omega)$, which takes the form
\begin{equation}\label{eq:clusterexpansion}
 \ln\Xi(\cP,\omega)=\sum_{\Gamma\in\cC}\omega(\Gamma).
\end{equation}
For $j\ge1$, its $j$th term is
\[
 L_j=\sum_{\substack{\Gamma\in\cC\\\|\Gamma\|=j}}\omega(\Gamma).
\]
The cluster expansion is an infinite series even when the polymer set is finite, as a cluster may repeat polymers. A sufficient condition for the convergence of the cluster expansion is
given by Koteck\'y and Preiss~\cite{kotecky1986cluster} in 1986.
\begin{theorem}[Koteck\'y--Preiss]\label{thm:KP}
Let $f,g:\cP\to[0,\infty)$ be two functions that satisfy
\[
 \sum_{\gamma\sim\gamma_0}|\omega(\gamma)|\exp\bigl(f(\gamma)+g(\gamma)\bigr)\le f(\gamma_0)
\]
for every $\gamma_0\in\cP$. Then the cluster expansion~\eqref{eq:clusterexpansion} converges absolutely.
Moreover, for a cluster $\Gamma\in\cC$, let
\(g(\Gamma)=\sum_{\gamma\in\Gamma}g(\gamma)\)
and write $\Gamma\sim\gamma_0$ if $\gamma\sim\gamma_0$ for some $\gamma\in\Gamma$. Then, for every $\gamma_0\in\cP$,
\begin{equation}\label{eq:KPtail}
 \sum_{\Gamma\in\cC,\ \Gamma\sim\gamma_0}|\omega(\Gamma)|\exp(g(\Gamma))
 \le f(\gamma_0).
\end{equation}
\end{theorem}

\section{Proof of Lemma~\ref{thm:local}: Counting families in a fixed star phase}\label{sec:local}
Throughout this section, fix an arbitrary center $x\in[n]$. Our goal is to count the intersecting families $\cF$ satisfying
\(
 |\cF\setminus\cS_x|\le\xi_2M.
\)
Recall from Section~\ref{sec:containers} that $H$ is the bipartite inclusion graph with parts
\[
 \mathcal U=\binom{[n]\setminus\{x\}}{k+r-1},\qquad
 \mathcal L=\binom{[n]\setminus\{x\}}{k-1},
\]
and every intersecting family $\cF$ is encoded by the independent set $A_x(\cF)\cup B_x(\cF)$ of $H$.
We say that an independent set $I$ of $H$ is \emph{valid} if $I\cap\mathcal U$ is $r$-intersecting. 
By Lemma~\ref{lem:encoding}, these are precisely the independent sets corresponding to intersecting families. Thus, to prove Lemma~\ref{thm:local}, it is equivalent to counting the number of valid independent sets with \[|I\cap\mathcal U|\le\xi_2M.\]

In the remainder of this section, we first introduce polymer models for independent sets of \(H\) and then state our key theorem, Theorem~\ref{thm:local-independent}, which estimates the polymer partition functions and thereby the number of independent sets.
We then establish convergence of the corresponding cluster expansions, bound their tails, and compute their leading terms.
Finally, we prove Theorem~\ref{thm:local-independent} and deduce Lemma~\ref{thm:local}.
All neighborhoods and closures are taken with respect to $H$, and we omit the subscript $H$.

\subsection{Polymer models for independent sets of $H$}
By Lemma~\ref{lem:encoding}, every independent set $I$ of $H$ can be constructed by first choosing its upper part $A=I\cap\mathcal U$, and then choosing its lower part $I\cap\mathcal L$ as \emph{any} subset of $\mathcal L\setminus N(A)$. 
Thus, for each fixed $A$, there are exactly $2^{M-|N(A)|}$
choices for $I$.
Accordingly, for
$\xi_2>0$ define
\[
 Z^{\ind}(\xi_2)
 :=\sum_{A\subseteq\mathcal U,\ |A|\le\xi_2 M}
   2^{-|N(A)|},
 \qquad
 Z^{\val}(\xi_2)
 :=\sum_{\substack{A\subseteq\mathcal U,\ |A|\le\xi_2 M\\
                  A\ r\text{-intersecting}}}
   2^{-|N(A)|}.
\]
Namely, $2^MZ^{\ind}(\xi_2)$ counts the independent sets $I$ of $H$ with $|I\cap \mathcal{U}|\le\xi_2M$, while
$2^MZ^{\val}(\xi_2)$ counts the valid ones. 

Set
\[
 a_0:=\binom{n-18}{k+r-1}.
\]
Define the ordinary polymer set by
\[
 \cP^{\ind}
 :=\{\gamma\subseteq\mathcal U:\gamma\ne\varnothing,\
   \gamma\text{ is 2-linked},\ |[\gamma]|\le a_0\},
\]
where $\gamma\sim\gamma'$ if $\gamma\cup\gamma'$ is 2-linked.
We also define a valid polymer set as
\[
 \cP^{\val}
 :=\{\gamma\in\cP^{\ind}:\gamma\text{ is }r\text{-intersecting}\},
\]
where $\gamma\sim\gamma'$ if $\gamma\cup\gamma'$ is 2-linked or
is not $r$-intersecting. 

We remark that the closure cutoff $|[\gamma]|\le a_0$ is needed to establish
convergence of the cluster expansion using Sapozhenko's container
method, as it provides the required expansion bound for the polymers.
Subsets $A\subseteq\mathcal U$ with a 2-linked component whose closure
exceeds $a_0$ will be accounted for separately in the proof of
Theorem~\ref{thm:local-independent}.

For both polymer sets, set
\[
 \omega(\gamma):=2^{-|N(\gamma)|}.
\]
Let $\sigma\in\{\ind,\val\}$. By the definition of adjacency, distinct polymers in a compatible
configuration $\Lambda\in\Omega_{\cP^\sigma}$ have disjoint
neighborhoods in $H$. Hence,
\begin{equation}\label{eq:factorization}
 \prod_{\gamma\in\Lambda}\omega(\gamma)=2^{-\sum_{\gamma\in\Lambda} \left|N(\gamma)\right|}
 =2^{-\left|N\left(\bigcup_{\gamma\in\Lambda}\gamma\right)\right|}.
\end{equation}
Note that in the valid model, adjacency also imposes the $r$-intersection condition between distinct components.

Recall that
\[
\Xi(\cP^\sigma,\omega)
=\sum_{\Lambda\in\Omega_{\cP^\sigma}}
\prod_{\gamma\in\Lambda}\omega(\gamma).
\]
Note that neither \(\Xi(\cP^\sigma,\omega)\) nor \(Z^\sigma(\xi_2)\) necessarily bounds the other: the former includes configurations with total size exceeding \(\xi_2M\), while the latter includes sets \(A\subseteq\mathcal U\) with a 2-linked component whose closure exceeds \(a_0\). Nevertheless, Theorem~\ref{thm:local-independent} shows that these differences are asymptotically negligible.

When $r=1$, we need one more polymer model to describe the typical structure.
We first state the definition of \emph{niceness} used in Theorem~\ref{thm:odd}
and Lemma~\ref{thm:local}.

\begin{definition}[BGLW~\cite{BaloghGarciaLiWagner2024}]\label{def:nice}
Let $n=2k+1$, let $\cF\subseteq\binom{[n]}k$ be intersecting,
and fix $x\in[n]$. We say that $\cF$ is \emph{nice with center $x$}
if every 2-linked component of $A_x(\cF)$ has size at most two.
We call $\cF$ \emph{nice} if it is nice with center $x$ for some
$x\in[n]$.
\end{definition}
Define
\[
 \cP_{\le2}^{\val}
 :=\{\gamma\in\cP^{\val}:|\gamma|\le2\},
\qquad
\text{and}
\qquad
 Z_{\le2}^{\val}(\xi_2)
:=\sum_{\Lambda\in\Omega_{\cP_{\le2}^{\val}},\  \|\Lambda\|\le\xi_2M}
       \prod_{\gamma\in\Lambda}\omega(\gamma).
\]
Thus, $2^MZ_{\le2}^{\val}(\xi_2)$ counts the intersecting families $\cF$ with $|\cF\setminus\cS_x|\le\xi_2M$ that are nice with center $x$.

For $\sigma\in\{\ind,\val\}$, let $\cC^\sigma$ be the set of clusters
of $(\cP^\sigma,\omega)$, and, for $j\ge1$, write $L_j^\sigma$ for
its $j$th cluster term. Let
\[
 \epsilon_r(k):=
 \begin{cases}
  k^{17}2^{-k},&r=1,\\
  2^{-k^2},&r=2,
 \end{cases}
\]
which will be the error term in the asymptotic formulas for the polymer partition functions.

We now state the main theorem of this section, which counts independent sets through the corresponding polymer models. 
\begin{theorem}\label{thm:local-independent}
There is an absolute $C_2>0$ such that, for every fixed
$0<\xi_2\le C_2$, every $r\in\{1,2\}$, and every
$\sigma\in\{\ind,\val\}$,
\begin{equation*}\label{eq:local-partition}
 Z^\sigma(\xi_2)=\Xi(\cP^\sigma,\omega)
   +O\!\left(\exp\bigl(-\Omega(\xi_2M/k)\bigr)\right)=\exp\left(L_1^\sigma+L_2^\sigma
            +O(\epsilon_r(k))\right).
\end{equation*}
If $r=1$, then also
\begin{equation*}\label{eq:nice-local-partition}
Z_{\le2}^{\val}(\xi_2)
 =\Xi(\cP_{\le2}^{\val},\omega)
  +O\!\left(\exp\bigl(-\Omega(\xi_2M/k)\bigr)\right)
 =\exp\left(L_1^{\val}+L_2^{\val}
           +O(k^{17}2^{-k})\right).
\end{equation*}
\end{theorem}
The implicit constants for error are absolute, while the threshold
for sufficiently large $k$ may depend on $\xi_2$.
\subsection{Convergence of the cluster expansion}
For every $\sigma\in\{\ind,\val\}$ and $\gamma\in\cP^\sigma$,
the closure cutoff $|[\gamma]|\le a_0=\binom{n-18}{k+r-1}$,
together with Proposition~\ref{prop:isoperimetry}, gives
\[
 \frac{|N(\gamma)|}{|[\gamma]|}
 \ge \prod_{i=0}^{r-1}\frac{k+i}{(n-18)-k+1-i}=
  \begin{cases}
 \displaystyle\frac{k}{k-16},&r=1,\\[2mm]
 \displaystyle\frac{k(k+1)}{(k-15)(k-16)},&r=2.
 \end{cases}
\]
Consequently, for sufficiently large $k$, every polymer satisfies
\begin{equation}\label{eq:localgap}
 1-\frac{|[\gamma]|}{|N(\gamma)|}\ge\frac{8r}{k}.
\end{equation}

We call a polymer $\gamma$ \emph{elementary} if
$|N(\gamma)|\le k^{12}$.
For non-elementary polymers, we can estimate their contribution
through the container lemma, as they satisfy the hypotheses of
Lemma~\ref{lem:reconstruct}.
For elementary polymers, Proposition~\ref{prop:isoperimetry}
gives a stronger expansion bound, so a more elementary counting
argument for connected subsets suffices, as in the proof of
Lemma~\ref{lem:small}.

Fix a sufficiently small absolute constant $\eta>0$. Define
\begin{equation}\label{def:ffun}
 f(\gamma):=\frac{\eta}{k}|\gamma|,
\end{equation}
and
\begin{equation}\label{eq:full-reserve}
 q(\gamma):=
 \begin{cases}
 |N(\gamma)|\ln2-6|\gamma|\ln k,
 &|N(\gamma)|\le k^{12},\\[2mm]
  \dfrac{c_0\ln2}{4}
 \bigl(|N(\gamma)|-|[\gamma]|\bigr),
 &|N(\gamma)|>k^{12},
 \end{cases}
\end{equation}
where $c_0>0$ is the constant in Lemma~\ref{lem:reconstruct}.

\begin{lemma}\label{lem:convergence}
For all sufficiently large $k$, every
$\sigma\in\{\ind,\val\}$, and every $\gamma_0\in\cP^\sigma$,
\[
 \sum_{\gamma\sim\gamma_0}
 \omega(\gamma)e^{f(\gamma)+q(\gamma)}
 \le f(\gamma_0).
\]
Consequently, the cluster expansion for
$\ln\Xi(\cP^\sigma,\omega)$ converges absolutely.
\end{lemma}

\begin{proof}
We first verify that $q$ is nonnegative, as required by
Theorem~\ref{thm:KP}.
For an elementary polymer $\gamma$, biregularity gives
$|\gamma|\le\rho|N(\gamma)|\le2k^{12}$.
Thus Proposition~\ref{prop:isoperimetry}(ii) gives
$|N(\gamma)|\ge d|\gamma|/200$, and hence
\begin{equation}\label{est:qfunction}
    q(\gamma)\ge|\gamma|(d\ln2/200-6\ln k)\ge0
\end{equation}
for sufficiently large $k$.
For non-elementary polymers, $q(\gamma)\ge0$ follows
from~\eqref{eq:localgap}.

For an elementary polymer $\gamma$, our choices of $f$ and $q$ give
\[
 \omega(\gamma)e^{f(\gamma)+q(\gamma)}
 =2^{-|N(\gamma)|}
  \exp\left(
    \frac{\eta|\gamma|}{k}
    +|N(\gamma)|\ln2-6|\gamma|\ln k
  \right)
 =e^{\eta|\gamma|/k}k^{-6|\gamma|}.
\]
Since $H^2[\mathcal U]$ has maximum degree at most
$\rho d^2\le2k^4$, Proposition~\ref{prop:connected-sets}
bounds the number of 2-linked $s$-sets containing a fixed
$u\in\mathcal U$ by $(2ek^4)^{s-1}$.
Taking $\eta\le1$, we therefore obtain
\begin{equation}\label{est:epolymer}
 \sum_{\substack{\gamma\in\cP^\sigma,\ u\in\gamma\\
                  |N(\gamma)|\le k^{12}}}
 \omega(\gamma)e^{f(\gamma)+q(\gamma)}
 \le
 \sum_{s\ge1}(2ek^4)^{s-1}e^{\eta s/k}k^{-6s}=
 \frac{e^{\eta/k}k^{-6}}
      {1-2e^{1+\eta/k}k^{-2}}
 \le2k^{-6}.
\end{equation}

For a non-elementary polymer $\gamma$, write
$a:=|[\gamma]|$ and $g:=|N(\gamma)|>k^{12}$.
Then our choices of $f$ and $q$ give
\[
 \omega(\gamma)e^{f(\gamma)+q(\gamma)}
 =2^{-g}\exp\left(
   \frac{\eta|\gamma|}{k}
   +\frac{c_0\ln2}{4}(g-a)
 \right)
 \le2^{-g}\exp\left(
   \frac{\eta g}{k}
   +\frac{c_0\ln2}{4}(g-a)
 \right),
\]
where we used $|\gamma|\le a\le g$.
By~\eqref{eq:localgap}, these polymers satisfy the hypotheses
of Lemma~\ref{lem:reconstruct}, which gives
$|\cG(a,g)|\le2^{g-c_0(g-a)}$. Thus, for each fixed $a,g$,
\[
 \log\left(
 \sum_{\substack{\gamma\in\cP^\sigma,\ |[\gamma]|=a\\
                 |N(\gamma)|=g>k^{12}}}
 \omega(\gamma)e^{f(\gamma)+q(\gamma)}
 \right)
 \le g-c_0(g-a)+\left(-g +\frac{\eta g}{k\ln2}
       +\frac{c_0}{4}(g-a)\right)
 =\frac{\eta g}{k\ln2}-\frac{3c_0}{4}(g-a)
 \le-\frac{c_0}{2}(g-a),
\]
where the last inequality follows from
$g-a\ge8rg/k$ and a sufficiently small choice of $\eta$.
Since $g>k^{12}$, we have $g-a>8rk^{11}$.
Summing over at most $(N+1)(M+1)=2^{O(k)}$ possible pairs $(a,g)$
therefore gives
\begin{equation}\label{est:nepolymer}
 \sum_{\substack{\gamma\in\cP^\sigma\\
                  |N(\gamma)|>k^{12}}}
 \omega(\gamma)e^{f(\gamma)+q(\gamma)}
 \le2^{-\Omega(k^{11})}.
\end{equation}

If $\gamma\cup\gamma_0$ is 2-linked, then $\gamma$ contains a vertex either in $\gamma_0$ or in its neighborhood in $H^2[\mathcal U]$; there are at most
$(\rho d^2+1)|\gamma_0|$ such vertices.
In the valid model, $\gamma$ and $\gamma_0$ are also adjacent if their union is not $r$-intersecting. In this case, $\gamma$ contains a vertex $u$ with $|u\cap v|<r$ for some $v\in\gamma_0$.
By~\eqref{eq:outside-degree}, there are at most $(rd/k)|\gamma_0|$ such vertices $u$.

Finally, summing~\eqref{est:epolymer} over these vertices and adding~\eqref{est:nepolymer}, we obtain, for every $\sigma\in\{\ind,\val\}$ and $\gamma_0\in\cP^\sigma$,
\[
 \sum_{\gamma\sim\gamma_0}
 \omega(\gamma)e^{f(\gamma)+q(\gamma)}
 \le O(|\gamma_0|k^4)\cdot 2k^{-6}+2^{-\Omega(k^{11})}
 \le\frac{\eta|\gamma_0|}{k}
 =f(\gamma_0)
\]
for sufficiently large $k$.
By Theorem~\ref{thm:KP}, the cluster expansion for both
$\ln\Xi(\cP^\sigma,\omega)$ converges absolutely.
\end{proof}

\subsection{Bounding the tail of the cluster expansion}
To further obtain the asymptotic formulas for the partition functions, we need a tail bound for the cluster expansion.
\begin{lemma}\label{lem:clustertail}
For every $\sigma\in\{\ind,\val\}$, the polymer model $(\cP^\sigma,\omega)$ satisfies
\begin{equation}\label{eq:tail}
 \sum_{\Gamma\in\cC^\sigma,\ \|\Gamma\|\ge3}|\omega(\Gamma)|=O(\epsilon_r(k)).
\end{equation}
When $r=1$, the same bound holds for the clusters of $(\cP_{\le2}^{\val},\omega)$.
\end{lemma}
\begin{proof}
We first show that for every polymer $\gamma\in \cP^\sigma$ with $|\gamma|\ge3$,
\begin{equation}\label{eq:reserve-three-p}
 q(\gamma)\ge3d\ln2-18\ln k-O(k^{r-1}).
\end{equation}
For an elementary polymer $\gamma$ with $3\le |\gamma|\le1000$,
since any two distinct vertices in $\mathcal U$ have at most $\binom{k+r-2}{k-1}
 \le k^{r-1}$ common
neighbors in $H$, we have
\begin{equation}\label{est:qfun-cn}
     q(\gamma)
 \ge\left(d|\gamma|-k^{r-1}\binom{|\gamma|}{2}\right)\ln2
     -6|\gamma|\ln k
 \ge 3(d\ln2-6\ln k)-O(k^{r-1}).
\end{equation}
For an elementary polymer $\gamma$ with $|\gamma|>1000$,
\eqref{est:qfunction} gives, for sufficiently large $k$,
\[
 q(\gamma)
 \ge|\gamma|\left(\frac{d\ln2}{200}-6\ln k\right)
 \ge\frac{|\gamma|d\ln2}{300}>3d\ln2.
\]
For a non-elementary polymer $\gamma$, \eqref{eq:localgap}
and $|N(\gamma)|>k^{12}$ give
\[
 q(\gamma)
 =\frac{c_0\ln2}{4}|N(\gamma)|\left(1-\frac{|[\gamma]|}{|N(\gamma)|}\right)
 \ge\frac{2c_0r\ln2}{k}|N(\gamma)|
 =\Omega(k^{11}).
\]
Since $d=O(k^2)$, this exceeds $3d\ln2$ for sufficiently large $k$. This proves~\eqref{eq:reserve-three-p}.

Write $q(\Gamma)=\sum_{\gamma\in\Gamma}q(\gamma)$. We next show that every cluster $\Gamma\in\cC^\sigma$ with $\|\Gamma\|\ge3$ satisfies
\begin{equation}\label{eq:reserve-three}
 q(\Gamma)\ge3d\ln2-18\ln k-O(k^{r-1}).
\end{equation}
By~\eqref{eq:reserve-three-p} and the nonnegativity of $q$, any cluster containing a polymer of size at least three satisfies~\eqref{eq:reserve-three}.
It remains to consider clusters consisting only of polymers of size $1$ or $2$.
Write $\Gamma=(\gamma_1,\ldots,\gamma_m)$, where $|\gamma_i|\in\{1, 2\}$.
For every $\gamma_i$, an argument similar to that in~\eqref{est:qfun-cn} gives
\[
 q(\gamma_i)\ge
 |\gamma_i|(d\ln2-6\ln k)-k^{r-1}\ln2.
\]
Using $\|\Gamma\|\ge3$ and taking $k$ sufficiently large, we obtain
\[
 q(\Gamma)=\sum_{i=1}^m q(\gamma_i)\ge\|\Gamma\|\bigl(d\ln2-6\ln k-k^{r-1}\ln2\bigr)\ge3d\ln2-18\ln k-3k^{r-1}\ln2,
\]
which proves~\eqref{eq:reserve-three}.

By Lemma~\ref{lem:convergence} and~Theorem~\ref{thm:KP}, for every singleton polymer $\{u\}\in\cP^\sigma$,
\[
 \sum_{\Gamma\in\cC^\sigma,\ \Gamma\sim\{u\}}|\omega(\Gamma)|\exp(q(\Gamma))
 \le f(\{u\})=\eta/k.
\]
Since every cluster is adjacent to a singleton polymer $\{u\}$
for some $u\in\mathcal U$, we have
\[
 \sum_{\Gamma\in\cC^\sigma}
 |\omega(\Gamma)|\exp(q(\Gamma))
 \le\sum_{u\in\mathcal U}f(\{u\})
 =\frac{\eta N}{k}.
\]
Together with~\eqref{eq:reserve-three}, this gives
\[
 \sum_{\Gamma\in\cC^\sigma,\ \|\Gamma\|\ge3}
 |\omega(\Gamma)|
 \le\frac{\eta N}{k}
 \exp\bigl(-3d\ln2+18\ln k+O(k^{r-1})\bigr)
 \le\eta Nk^{17}2^{-3d+O(k^{r-1})}
 =
 \begin{cases}
 O(k^{17}2^{-k}),&r=1,\\
 O(2^{-k^2}),&r=2,
 \end{cases}
\]
where we used $N\le2^{2k+r-1}$ and the formulas for $d$.

When $r=1$, every cluster of
$(\cP_{\le2}^{\val},\omega)$ belongs to $\cC^{\val}$
and has the same weight in both models.
The same bound therefore holds for this restricted model.
\end{proof}

\subsection{Computing the first two terms}\label{est:cluster}
Lemma~\ref{lem:clustertail} shows that the leading contribution
to the cluster expansion comes from its first two terms.
We next compute explicit formulas for $L_1^\sigma$ and
$L_2^\sigma$ for $\sigma\in\{\ind,\val\}$.

\medskip\noindent\textbf{Computing $L_1^\sigma$.} Every polymer of size one is a single vertex of $\mathcal U$. There are $N$ such polymers, and each has weight $2^{-d}$. There is only one type of cluster of size one, consisting of a single polymer of size one, with Ursell function $1$. Therefore,
\[
 L_1^{\ind}=L_1^{\val}=N2^{-d}.
\]
In particular, for $r=2$ and $\sigma\in\{\ind,\val\}$,
\[
 L_1^\sigma=\binom{2k+1}{k}2^{-k(k+1)/2}
 =2^{-k^2/2+O(k)}=o(1).
\]

\medskip\noindent\textbf{Computing $L_2^\sigma$.} 
A polymer of size two consists of two distinct vertices $u,v\in\mathcal U$ with a common neighbor in $H$, i.e., $|u\cap v|\ge k-1$.
For each integer $k-1\le i \le k+r-2$, let
\[
 D_i:=\binom{k+r-1}{i}\binom{k}{k+r-1-i},
 \qquad
 h_i:=\binom{i}{k-1}.
\]
For every $u\in\mathcal U$, there are exactly $D_i$
choices of $v$ with $|u\cap v|=i$, so there are $ND_i/2$
such polymers.
The vertices $u,v$ have exactly $h_i$ common neighbors,
and hence each such polymer has weight $2^{-2d+h_i}$.
For sufficiently large $k$, all these pairs are $r$-intersecting and therefore belong to both
$\cP^{\ind}$ and $\cP^{\val}$.

There are two types of clusters of size two. 
The first consists of an ordered pair of adjacent singleton
polymers, which may be equal, with Ursell function $-1/2$.
In the ordinary model $(\cP^{\ind},\omega)$, there are $N$ such clusters with
equal singleton polymers and $N\sum_{i=k-1}^{k+r-2}D_i$
with distinct singleton polymers.
The second type consists of a single polymer of size two, with Ursell function $1$. Combining these contributions gives
\[
 L_2^{\ind}=-\frac12\left(N+N\sum_{i=k-1}^{k+r-2}D_i\right)2^{-2d}
   +\frac N2\sum_{i=k-1}^{k+r-2}D_i2^{-2d+h_i}=\frac N22^{-2d}\left(-1+\sum_{i=k-1}^{k+r-2}D_i(2^{h_i}-1)\right).
\]
In the valid model $(\cP^{\val},\omega)$, two singleton polymers $\{u\}$ and $\{v\}$
are also adjacent when $|u\cap v|<r$.
For each $u\in\mathcal U$, there are exactly
$\binom{k+r-1}{r-1}=rd/k$ such vertices $v$.
For $k\ge4$, $u$ and $v$ have no common neighbor in $H$,
so $\{u\}$ and $\{v\}$ are not adjacent in the ordinary model.
Therefore, there are $Nrd/k$ additional ordered pairs of adjacent
singleton polymers, each with Ursell function $-1/2$.
Hence
\[
 L_2^{\val}
 =L_2^{\ind}-\frac12\cdot\frac{Nrd}{k}\,2^{-2d}
 =L_2^{\ind}-\frac{Nrd}{2k}2^{-2d}.
\]
For $r=1$, this gives
\[
 L_2^{\ind}=\frac N2(k^2-1)2^{-2k},\qquad
 L_2^{\val}=\frac N2(k^2-2)2^{-2k}.
\]
For $r=2$, we have for $\sigma\in\{\ind,\val\}$,
\[
 L_2^\sigma=O\left(Nk^2 2^{k-2d}\right)=2^{-k^2+O(k)}=o(1).
\]

\subsection{Proof of Theorem~\ref{thm:local-independent}}

For $\sigma\in\{\ind,\val\}$, Lemmas~\ref{lem:convergence} and~\ref{lem:clustertail} give
\[
\Xi(\cP^\sigma,\omega)
 =\exp\left(L_1^\sigma+L_2^\sigma+O(\epsilon_r(k))\right).
\]
It remains to relate the polymer partition function $\Xi(\cP^\sigma,\omega)$ to the sum $Z^\sigma(\xi_2)$. 
Recall that there are two types of discrepancy: $\Xi(\cP^\sigma,\omega)$ includes compatible configurations $\Lambda$ with $\|\Lambda\|>\xi_2M$, whereas $Z^\sigma(\xi_2)$ includes sets $A\subseteq\mathcal U$ with a 2-linked component
whose closure exceeds $a_0$. 
We will show that both discrepancies are negligible.
The following lemma controls the first type; the second is handled directly in the proof of the theorem.

\begin{lemma}\label{lem:defecttail}
For every fixed $\xi_2>0$ and every $\sigma\in\{\ind,\val\}$,
\[
 \sum_{\Lambda\in\Omega_{\cP^\sigma},\ \|\Lambda\|>\xi_2 M}
 \prod_{\gamma\in\Lambda}\omega(\gamma)
 \le \exp\bigl(-\Omega(\xi_2M/k)\bigr).
\]
\end{lemma}
\begin{proof}
Since $q(\gamma)\ge0$, summing~\eqref{est:epolymer}
over $u\in\mathcal U$ and adding~\eqref{est:nepolymer} gives
\begin{equation}\label{eq:tiltedsingle}
 S_\sigma:=
 \sum_{\gamma\in\cP^\sigma}\omega(\gamma)e^{f(\gamma)}
 \le2Nk^{-6}+2^{-\Omega(k^{11})}
 =o(M/k),
\end{equation}
where we used $N=\rho M=O(M)$.
By the definition of $f$ (see~\eqref{def:ffun}),
\[
 \sum_{\gamma\in\Lambda}f(\gamma)
 =\frac{\eta}{k}\|\Lambda\|.
\]
Dropping the compatibility condition and using $1+t\le e^t$,
we therefore obtain
\[
 \sum_{\Lambda\in\Omega_{\cP^\sigma}} e^{\eta\|\Lambda\|/k}\prod_{\gamma\in\Lambda}\omega(\gamma)
 = \sum_{\Lambda\in\Omega_{\cP^\sigma}}
 \prod_{\gamma\in\Lambda}\omega(\gamma)e^{f(\gamma)}
 \le\prod_{\gamma\in\cP^\sigma}
 \bigl(1+\omega(\gamma)e^{f(\gamma)}\bigr)
 \le e^{S_\sigma}.
\]
Consequently,
\[
 \sum_{\substack{\Lambda\in\Omega_{\cP^\sigma},\ \|\Lambda\|>\xi_2M}}\prod_{\gamma\in\Lambda}\omega(\gamma)
 \le e^{-\eta\xi_2M/k}\sum_{\Lambda\in\Omega_{\cP^\sigma}}e^{\eta\|\Lambda\|/k}
 \prod_{\gamma\in\Lambda}\omega(\gamma)
 \le\exp\left(-\frac{\eta\xi_2M}{k}+o\left(\frac Mk\right)\right)
 \le\exp\left(-\frac{\eta\xi_2M}{2k}\right)
\]
for sufficiently large $k$.
This proves the lemma.
\end{proof}

We now have all the ingredients for the proof of Theorem~\ref{thm:local-independent}.

\begin{proof}[Proof of Theorem~\ref{thm:local-independent}]
Choose $C_2\in(0,1/4)$ sufficiently small that $2H_2(C_2)<2^{-19}$.
Fix $r\in\{1,2\}$, $\sigma\in\{\ind,\val\}$, and $0<\xi_2\le C_2$.  
Lemmas~\ref{lem:convergence} and~\ref{lem:clustertail} give
\begin{equation}\label{eq:fullexpansion}
\Xi(\cP^\sigma,\omega)
 =\exp\left(L_1^\sigma+L_2^\sigma+O(\epsilon_r(k))\right).
\end{equation}

Next we estimate $|Z^\sigma(\xi_2)-\Xi(\cP^\sigma,\omega)|$. Consider sets $A\subseteq\mathcal U$ with
$|A|\le\xi_2M$ and $|[A]|>a_0$.
Note that $a_0/N\to2^{-17}$ and $N= \rho M$. Then, for sufficiently large $k$,
\[
a_0/\rho\ge2^{-18}M.
\]
By biregularity and Proposition~\ref{lem:binomial-estimates}, we have
\[
 |N(A)|\ge\frac{|[A]|}{\rho}
 >\frac{a_0}{\rho}\ge2^{-18}M,
\]
and the number of
such sets $A$ is at most
\[
 \sum_{j\le\xi_2M}\binom Nj
 \le2^{\rho M H_2(\xi_2/\rho)}
 \le2^{2MH_2(C_2)} < 2^{2^{-19}M}.
\]
Every set $A$ counted by $Z^\sigma(\xi_2)$ but not represented in $\Xi(\cP^\sigma,\omega)$ must satisfy $|A|\le\xi_2M$ and $|[A]|>a_0$.
Hence
\begin{equation}\label{eq:Zupper}
 Z^\sigma(\xi_2)-\Xi(\cP^\sigma,\omega)
 \le
 \sum_{\substack{A\subseteq\mathcal U\\
                  |A|\le\xi_2M,\ |[A]|>a_0}}
 2^{-|N(A)|}
 < 2^{2^{-19}M-2^{-18}M}
 =2^{-\Omega(M)}.
\end{equation}

Conversely, every configuration $\Lambda\in\Omega_{\cP^\sigma}$ counted by $\Xi(\cP^\sigma,\omega)$ but not represented in $Z^\sigma(\xi_2)$ must satisfy $\|\Lambda\|>\xi_2M$. Lemma~\ref{lem:defecttail} therefore gives
\begin{equation}\label{eq:Zlower}
\Xi(\cP^\sigma,\omega)- Z^\sigma(\xi_2)\le\exp(-\Omega(\xi_2M/k)).
\end{equation}
Both partition functions are at least $1$, and the comparison error is $o(\epsilon_r(k))$ for each fixed $\xi_2>0$.
Combining~\eqref{eq:Zupper}, \eqref{eq:Zlower}, and~\eqref{eq:fullexpansion} gives
 \[
 Z^\sigma(\xi_2)=\Xi(\cP^\sigma,\omega)
   +O\!\left(\exp\bigl(-\Omega(\xi_2M/k)\bigr)\right)=\exp\left(L_1^\sigma+L_2^\sigma
            +O(\epsilon_r(k))\right),\] 
which proves the first part of Theorem~\ref{thm:local-independent}.

For $r=1$, the model on $\cP_{\le2}^{\val}$ retains all clusters of total size at most two. Lemmas~\ref{lem:convergence} and~\ref{lem:clustertail}
therefore give
\[
 \Xi(\cP_{\le2}^{\val},\omega)
 =\exp\left(L_1^{\val}+L_2^{\val}+O(k^{17}2^{-k})\right).
\]
Since all polymers have size at most two, their closures have size at most $2\rho d \le a_0$ for sufficiently large $k$.
Moreover, the proof of Lemma~\ref{lem:defecttail} applies unchanged
to the restricted model, giving
\[
0\le \Xi(\cP_{\le2}^{\val},\omega)- Z_{\le2}^{\val}(\xi_2)\le\exp(-\Omega(\xi_2M/k)).
\]
Hence, we obtain
\[
Z_{\le2}^{\val}(\xi_2)
 =\Xi(\cP_{\le2}^{\val},\omega)
  +O\!\left(\exp\bigl(-\Omega(\xi_2M/k)\bigr)\right)
 =\exp\left(L_1^{\val}+L_2^{\val}
           +O(k^{17}2^{-k})\right),
\]
which proves the second part of Theorem~\ref{thm:local-independent}.
\end{proof}

\subsection{Deduction of Lemma~\ref{thm:local}}

\begin{proof}[Proof of Lemma~\ref{thm:local}]
Let $C_2$ be the absolute constant supplied by Theorem~\ref{thm:local-independent}, and fix $0<\xi_2\le C_2$ and $x\in[n]$.

For $r=2$, the computations in Section~\ref{est:cluster} give
\[
 L_1^{\val}+L_2^{\val}+O(\epsilon_2(k))\le2^{-k^2/3}
\]
for sufficiently large $k$.
By Theorem~\ref{thm:local-independent}, the number of nontrivial intersecting families $\cF\subseteq\binom{[n]}k$ satisfying $|\cF\setminus\cS_x|\le\xi_2M$ is at most
\[
2^MZ^{\val}(\xi_2)-2^M \le 2^M\left(\exp(2^{-k^2/3})-1\right)\le2^{M-c_2k^2},
\]
for some absolute $c_2>0$. This proves the first part of Lemma~\ref{thm:local} for $n=2k+2$.

For $r=1$, the computations in Section~\ref{est:cluster} give
\[
 L_1^{\val}=N2^{-k},\qquad L_2^{\val}=\frac N2(k^2-2)2^{-2k}.
\]
In particular, $L_1^{\val}\sim2^k/\sqrt{\pi k}$ and $L_2^{\val}\sim k^{3/2}/(2\sqrt\pi)$. 
By Theorem~\ref{thm:local-independent},
\[
 \#\{\cF\text{ intersecting}:|\cF\setminus\cS_x|\le\xi_2M\}
 =2^MZ^{\val}(\xi_2)=2^M\exp\left(N2^{-k}+\frac N2(k^2-2)2^{-2k}
 +O(k^{17} 2^{-k})\right).
\]
Since there are at most $n2^M$ trivial families and
$N2^{-k}\sim2^k/\sqrt{\pi k}$, removing the trivial families
does not change the asymptotic formula, including its
$O(k^{17}2^{-k})$ error term. This proves the first part of Lemma~\ref{thm:local} for $n=2k+1$.

Finally, the second part of Theorem~\ref{thm:local-independent} shows that the number of intersecting families $\cF$ that are nice with center $x$ and satisfy $|\cF\setminus\cS_x|\le\xi_2M$ is
\[
 2^MZ_{\le2}^{\val}(\xi_2)
 =2^M\exp\left(N2^{-k}+\frac N2(k^2-2)2^{-2k}
 +O(k^{17}2^{-k})\right).
\]
Removing the trivial families again leaves this formula
unchanged, proving the second part of Lemma~\ref{thm:local}.
\end{proof}

\section{Concluding Remarks}
It would be interesting to study a weighted version of the problem.
For fugacity $\lambda>0$, define
\[
Z_{n,k}(\lambda)=\sum_{\mathcal F}\lambda^{|\mathcal F|},
\]
where the sum ranges over all intersecting families
$\mathcal F\subseteq\binom{[n]}k$. Equivalently,
$Z_{n,k}(\lambda)$ is the hard-core partition function of the
Kneser graph $\KG(n,k)$. The case $\lambda=1$ recovers the
enumeration problem considered here. 

For general $n>2k$,  a natural question is to determine the
asymptotic behavior of $Z_{n,k}(\lambda)$ and the typical structure
of an intersecting family sampled with probability proportional
to $\lambda^{|\mathcal F|}$.
For $n\ge 2k+2$, one may ask for which ranges of $\lambda$ we have
\[
Z_{n,k}(\lambda)=(n+o(1))(1+\lambda)^M,
\qquad M=\binom{n-1}{k-1},
\]
and a typical family is contained in a unique full star.
In particular, for $n=2k+2$, we believe that our approach can be
adapted to establish this conclusion whenever
$\lambda\ge(\log k)^C/k$, for some absolute constant $C>0$.
For $n=2k+1$, typical families already contain defects outside
a full star at $\lambda=1$, motivating the study of how these
defects depend on the fugacity and when they become negligible.
More generally, it would be interesting to locate the transitions
between containment in a full star, concentration near a star
with defects, and other possible typical structures.

\bigskip
\noindent\textbf{Declaration of AI use.} The author proposed using Sapozhenko’s container method, particularly in Step 1. ChatGPT (5.6 Sol) was then used to discuss details of the argument, perform calculations, optimize parameters, assist with the proofs of several technical lemmas, and prepare an initial draft. ChatGPT (6-Astra) was subsequently used to improve the paper’s presentation and writing at the author’s direction. The author takes full responsibility for the content of the paper.

\bibliographystyle{abbrv}
\bibliography{typical_intersecting_2k_plus_2}

\end{document}